\documentclass[oneside, 11pt]{amsart}
\usepackage{amsmath}
\usepackage{amsfonts}
\usepackage{graphics}
\usepackage{epsfig}
\usepackage{amssymb}
\usepackage{amscd}
\usepackage[all]{xy}
\usepackage{latexsym}
\usepackage{graphicx}
\usepackage{multirow}
\usepackage{geometry}
\usepackage{color}
\usepackage[usenames,dvipsnames]{pstricks}
\usepackage{epsfig}
\usepackage{pst-grad} 
\usepackage{pst-plot} 
\usepackage{subfig}
\usepackage{float}
\usepackage{hyperref}
\usepackage{stmaryrd}

\newtheorem{thm}{Theorem}[section]

\newtheorem{lem}[thm]{Lemma}
\newtheorem{prop}[thm]{Proposition}
\newtheorem{conj}[thm]{Conjecture}
\newtheorem{quest}[thm]{Question}

\theoremstyle{definition}
\newtheorem{Def}[thm]{Definition}
\newtheorem{rem}[thm]{Remark}

\newtheorem{ex}{Example}[section]

\newtheorem{step}{Step}

\numberwithin{equation}{subsection}

\def\Hom{{\text{\rm{Hom}}}}
\def\End{{\text{\rm{End}}}}

\def\tr{{\text{\rm{tr}}}}
\def\rchi{{\hbox{\raise1.5pt\hbox{$\chi$}}}}
\def\Aut{{\text{\rm{Aut}}}}

\def\isom{\cong}

\def\tensor{\otimes}
\def\dsum{\oplus}
\def\Ker{{\text{\rm{Ker}}}}

\def\a{\alpha}
\def\b{\beta}
\def\lam{\lambda}

\def\Jac{{\text{\rm{Jac}}}}

\def\Jac{{\text{\rm{Jac}}}}
\def\Prym{{\text{\rm{Prym}}}}
\def\Dol{{\text{\rm{Dol}}}}
\def\deR{{\text{\rm{deR}}}}
\def\Betti{{\text{\rm{Betti}}}}
\def\rank{{\text{\rm{rank}}}}
\def\Gauge{{\text{\rm{Gauge}}}}
\def\ad{{\text{\rm{ad}}}}

\newcommand{\bea}{\begin{eqnarray}}
\newcommand{\eea}{\end{eqnarray}}
\newcommand{\be}{\begin{equation}}
\newcommand{\ee}{\end{equation}}

\newcommand{\Mbar}{{\overline{\mathcal{M}}}}
\newcommand{\bA}{{\mathbb{A}}}
\newcommand{\bP}{{\mathbb{P}}}
\newcommand{\bC}{{\mathbb{C}}}

\newcommand{\bF}{{\mathbb{F}}}
\newcommand{\bL}{{\mathbb{L}}}
\newcommand{\bH}{{\mathbb{H}}}

\newcommand{\bR}{{\mathbb{R}}}

\newcommand{\bZ}{{\mathbb{Z}}}

\newcommand{\cM}{{\mathcal{M}}}

\newcommand{\cD}{{\mathcal{D}}}

\newcommand{\cH}{{\mathcal{H}}}

\newcommand{\cO}{{\mathcal{O}}}

\newcommand{\cU}{{\mathcal{U}}}

\newcommand{\la}{{\langle}}
\newcommand{\ra}{{\rangle}}
\newcommand{\half}{{\frac{1}{2}}}

\newcommand{\rar}{\rightarrow}
\newcommand{\lrar}{\longrightarrow}

\title{A journey from the Hitchin section to the oper moduli}

\author{Olivia Dumitrescu}
\address{
Olivia Dumitrescu:
Pierce Hall 209\\
Central Michigan University\\
Mount Pleasant 48859, Michigan}
\address{and Simion Stoilow Institute of Mathematics\\
Romanian Academy\\
21 Calea Grivitei Street\\
010702 Bucharest, Romania}
\email{dumit1om@cmich.edu}

\thanks{The author is a member of 
the Simion Stoilow Institute of Mathematics of the 
Romanian Academy.}



\begin{document}

\maketitle

\centerline{\emph{To Anthony}}

\begin{abstract} 
This paper provides an introduction to
the mathematical notion of \emph{quantum curves}. 
We start with a concrete example arising from a 
graph enumeration problem. We then 
develop a theory of quantum curves associated
with Hitchin spectral curves. A conjecture of
Gaiotto, which predicts a new construction of
opers from a Hitchin spectral curve,
is explained.
We give a
step-by-step detailed description of the proof
of the conjecture  for the case of
rank $2$ Higgs bundles. Finally, we identify the two concepts of \textit{quantum curve} arising from the topological recursion formalism with the limit oper of Gaiotto's conjecture.

\end{abstract}

\setlength\intextsep{0pt}

\tableofcontents

\section{Introduction}

Mathematical research is  a journey. We start from
one place, often a remote place nobody cares. 
Guided by  mysteries one after another, we 
arrive at a place we have never imagined. We then 
suddenly realize that many people have come to
the same place, starting from totally different
origins. 

These are the lectures that the author has 
delivered in the last few years in many places of 
the world. They are meant to be an introduction to 
the notion of \emph{quantum curves}. Yet the
honest feeling that the author has now is that
these are more a record of how her 
understanding of  quantum curves has evolved.
The mathematics of quantum curves itself has been 
changing over the years. We have started from one 
place, based on what is known as 
\emph{topological recursion}. When we have 
arrived at the current position, we find 
ourselves  dealing with
\emph{opers}. 

The notion of \textit{quantum curves} was conceived in string theory by Aganagic, Dijkgraaf, Gukov, Hollands, Klemm, Marino, Sulkowski, Vafa, and others \cite{ADKMV,DHS,DHSV,GS}. We are far from 
establishing a complete theory at this moment. 
Yet we hope these lectures give a snapshot of
what is understood in the mathematics community
now, at least one 
of the many sides of the story of quantum curves.

This paper is organized as follows. In Section \ref{section 2}, we start from a simple question in enumerative geometry, and obtain the 
Dijkgraaf-Verlinde-Verlinde formula
\cite{DVV} for intersection numbers of $\psi$-classes on moduli space of stable curves $\Mbar_{g,n}$. More precisely, in Section \ref{eca cat}, we use the \textit{edge contraction} operations of ribbon graphs to generalize a count of graphs, and establish a recursion of \textit{Catalan numbers} of arbitrary genera. Then in Section \ref{GWP}, we present how the Laplace transform of the recursion of Catalan numbers surprisingly gives the DVV formula for intersection numbers on $\Mbar_{g,n}$. By the WKB analysis the recursion relation becomes equivalent to the \textit{quantization of the spectral curve} of the Catalan numbers. In Section \ref{thm:ECF}, we present how the same set of edge contraction operations on ribbon graphs give the \emph{cut-and-join equations} for orbifold Hurwitz numbers. 

We start with  presenting an introduction to the geometry of the Hitchin moduli spaces of holomorphic Higgs bundles and connections
in Section \ref{section 4}.  We then generalize the quantization theorem of Catalan recursion \ref{thm MS}, replacing the concept of spectral curves of Section \ref{spec curve} by the framework of Hitchin spectral curves. More precisely, following \cite{OD8, OD12,OD17},  we present the quantization results of spectral curves  for holomorphic and meromorphic Higgs bundles of rank $2$. Here, the algebro-geometric technique presented in Section \ref{spectral curve} was indispensable in quantizing \emph{singular} Hitchin spectral curves \cite{OD12}.

In Section \ref{section 5}, using the  work of Gunning \cite{Gun}, we propose to identify the two concepts: \textit{quantum curves} and \textit{opers}. The new idea of quantization in these notes is based on a recent solution \cite{OD20} of a conjecture due to the physicist  Gaiotto \cite{G}, presented in Sections 6 and 7. 

\subsection{Acknowledgments}

The author would like to express her gratitude to the organizers of \emph{String-Math 2016} held in 
Coll\`ege de France, Paris, and to the Institute Henri Poincar\'e, for their hospitality. These lecture notes grew out from the author's
paper \cite{OD20} in collaboration with L. Fredrickson, G. Kydonakis, R. Mazzeo, M. Mulase, and A. Neiztke, that solves a conjecture of Davide Gaiotto \cite{G}. This work was initiated at the AIM workshop, ``New perspectives on spectral data for Higgs bundles.'' The author also thanks the organizers of the workshop, in particular  Philip Boalch and Laura Schaposnik, for motivating interest in this problem by posing the question which led to this analysis.

The author is deeply indebted to Motohico Mulase for his generosity in mathematical discussions, enthusiasm, passion and encouragement that stimulated our collaboration throughout the years. This work could not have been produced without his support, for which the author would like to express all her gratitude.

The research of the author was supported by a grant from the Max-Planck Institute for Mathematics, Bonn. These lectures are based on a
collaboration  and discussions of the author with Motohico Mulase that took place in 2016 at the  Max-Planck Institute for Mathematics, Bonn,   and the Institute of Mathematics ``Simion Stoilow'' in
Bucharest. These lectures will be continued in \cite{OD21, OD22}.

\section{Enumeration of ribbon graphs}\label{section 2} 

\subsection{A combinatorial model for the moduli space of curves $\cM_{g,n}$}

\allowdisplaybreaks
\textit{Ribbon graphs}
are combinatorial objects  first used by G.~'t Hooft \cite{tH}  in quantum gauge theory, and later 
by Kontsevich
\cite{K1992}
in random matrix theory as the first approach to Gromov-Witten theory. They appeared independently in the work of Grothendieck \cite{Gro}
and are also known as \textit{dessins d'enfants}. A ribbon graph as a graph has a \textit{cyclic ordering} of the set of incident half-edges at each vertex and labeled faces. A ribbon graph embeds into an oriented
compact topological surface of type $(g,n)$, where $g$ represents the genus of the 
surface and $n$ the number of marked points
corresponding to the faces of the ribbon graph.
\begin{figure}[h]
\includegraphics[width=0.7\linewidth]{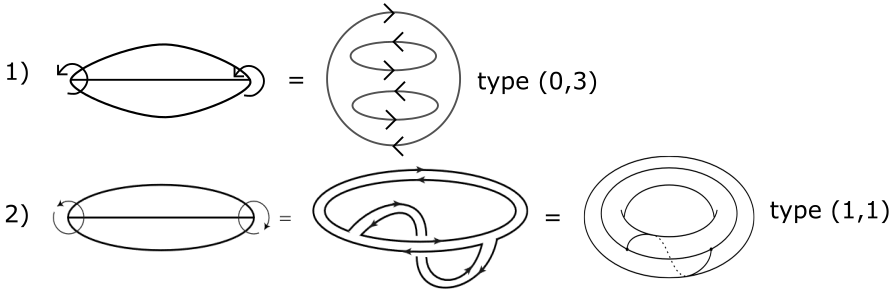}
\caption{}
\end{figure}

Decorating ribbon graphs by a positive real number on each edge fixes a topological surface of type $(g,n)$ together with \textit{a complex structure} on it. Introduce the total space of graphs as an orbifold parametrizing decorated ribbon graphs of a given
topological type $(g,n)$:
$$\mathcal{R}_{g,n} = \coprod_{\substack{\Gamma {\text{ boundary labeled}} \\ {\text{ribbon graph}}\\
{\text{of type }} (g,n) \\ {\text{and $e(\Gamma)$ edges}}}}
\frac{\bR_+ ^{e(\Gamma)}}{\Aut (\Gamma)}$$
\noindent
The space $\mathcal{R}_{g,n}$ is a smooth orbifold (see \cite[Section~3]{MP1998} and \cite{STT}). The combinatorial model of moduli space was constructed by  Thurston \cite{STT}, Harer \cite{Harer, HZ}, Mumford \cite{Mumford}, and Strebel \cite{Strebel} (cf. \cite{MP1998}). There exists an orbifold isomorphism between the total space of graphs  of type $(g,n)$ and the product of $ \bR_+^n$ and the moduli space $\cM_{g,n}$ of smooth algebraic curves of genus $g$ with $n$ marked points:
\be \label{orbifold iso}
\mathcal{R}_{g,n}\cong \cM_{g,n}\times \bR_+^n.
\ee

The isomorphism 
\eqref{orbifold iso} gives a 
cell-decompositions of the 
moduli space  $\cM_{g,n}$ for each
choice of $p\in \bR_+^n$,
and generalized Catalan numbers are related to a count of \textit{lattice points} in each cell-decomposition
for $p\in \bZ_+^n$. The isomorphism \eqref{orbifold iso} enables us to use the combinatorial model for the study of  topology of
$\cM_{g,n}$  via ribbon graphs and their geometry. Starting from a count of graphs,  or the 
number of orbi-cells in  $\mathcal{R}_{g,n}$, the corresponding enumerative problem on
$\cM_{g,n}$ surprisingly becomes the intersection 
numbers of the $\psi$ classes on $\Mbar_{g,n}$.

\subsection{The Combinatorics of Catalan recursion}\label{recursion}

In combinatorics, the Catalan numbers form a sequence of natural numbers that occur in various counting problems for recursively defined objects. They also appear in nature, and have more than twenty alternative definitions. To extend one of these interpretations we define the generalized Catalan number to count a number of graphs on a Riemann surface of genus $g$ with $n$ marked points. We define a \textit{cell graph} to be a ribbon graph with labeled vertices. We introduce the generalized Catalan numbers, $C_{g,n}(\mu_1,\ldots, \mu_n)$, as the \textit{number of cell graphs} of type $(g,n)$ with an outgoing arrow and degree $\mu_i$ at each vertex $i$. In Figure \eqref{RG}, we give an example of a vertex of degree $7$.

\begin{figure}[h]
\centering
\subfloat[]{{\label{RG}\includegraphics[width=.40\textwidth]{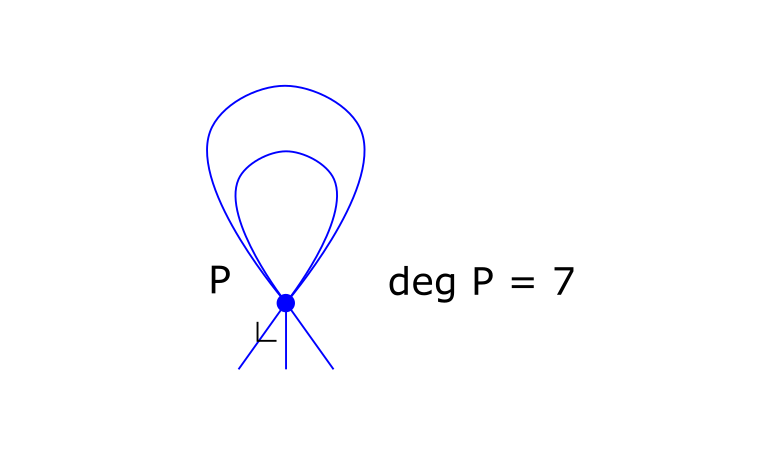}}}%
\subfloat[]{{\label{catalan RG}\includegraphics[width=.40\textwidth]{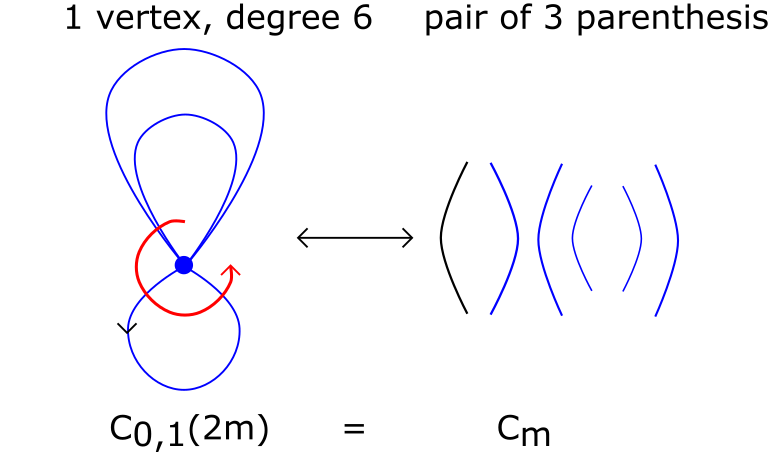}}}%
\caption{}
\end{figure}

\begin{ex}\label{cat}
For $(g,n)=(0,1)$, the Catalan numbers 
$C_m$ count graphs on a Riemann sphere with one vertex (so all edges are loops). We recall the classical definition of Catalan numbers counting the number of expressions containing $m$ pairs of parentheses which are correctly matched.
It is easy to see that $C_{0,1}(2m)=C_m$. In Figure \eqref{catalan RG} we provide an example explaining why each cell graph corresponds to a pair of parenthesis for $m=3$.
\end{ex}

Since for the classical Catalan numbers there is the recursion $C_m=\sum_{a+b=m-1}C_a\cdot C_b$, we expect to find a similar recursion for the generalized Catalan number $C_{g,n}(\mu_1,\ldots, \mu_n)$.
\subsection{Edge contraction operation and Catalan recursion}\label{eca cat}
\begin{thm}[Theorem 3.2, \cite{OD6}, \cite{WL}]\label{catalan}
For $2g-2+n\geq 0$, $n\geq 1$, the generalized Catalan numbers satisfy the following recursion
$$C_{g,n}(\mu_1,\dots,\mu_n)=
\sum_{j=2} ^n \mu_j \cdot C_{g,n-1}
(\mu_1+\mu_j-2,\mu_2,\dots,\widehat{\mu_j},
\dots,
\mu_n)+
$$
$$
+
\sum_{\a+\b = \mu_1-2}
\Bigg[
C_{g-1,n+1}(\a,\b,\mu_2,\dots,\mu_n)
+
\sum_{\substack{g_1+g_2=g\\I\sqcup J=\{2,\dots,n\}}}
C_{g_1,|I|+1}(\a,\mu_I)\; \cdot C_{g_2,|J|+1}(\b,\mu_J)
\Bigg].
$$
\end{thm}
\begin{proof} Starting from a cell graph with an arrowed edge at vertex $1$ we will contract this edge to a point, and we call this an \textbf{Edge Contraction Operation}.  We distinguish two cases.

Case 1. The arrowed edge connects vertex 1 of degree $\mu_1$ with the vertex $j > 1$ of degree $\mu_j$. We contract the edge and we join the two vertices $1$ and $j$ together as shown in figure below.
The resulting graph has one less vertex, but the genus is the same, the degree of the newly created vertex is $\mu_{1} + \mu_{j} - 2$; we mark the edge that was immediately
counterclockwise of the contracted edge, as indicated in Figure \ref{edge}.

\begin{figure}[h]
\centering
\subfloat[]{{\label{edge}\includegraphics[width=.45\textwidth]{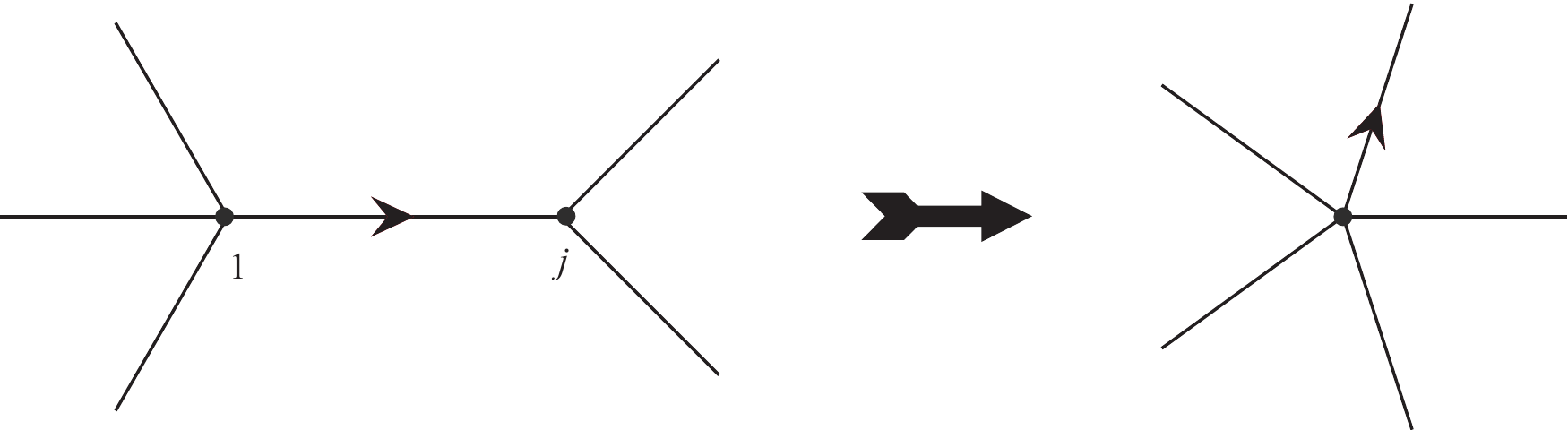}}}%
\qquad
\subfloat[]{{\label{loop}\includegraphics[width=.45\textwidth]{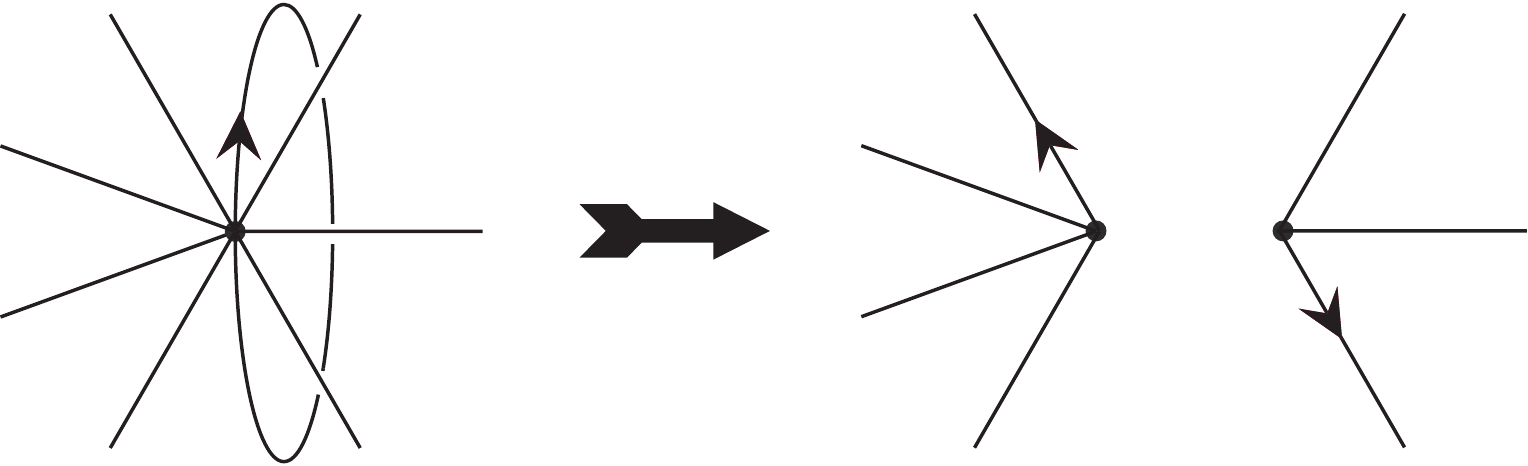}}}%
\caption{}
\end{figure}

Case 2. The arrowed edge is a loop attached to vertex $1$, then we remove this loop from the cell graph, and separate the vertex into two vertices. The loop classifies all other incident edges into two groups with $\alpha$ edges in one group, and $\beta$ the rest. Then $\alpha + \beta = \mu_{1} - 2$, and we create two vertices of degrees $\alpha$ and $\beta$ as in the Figure \ref{loop}.

\end{proof}

\subsection{Spectral curve for the Catalan recursion}\label{spec curve}

\begin{quest}
What is the mirror dual of Catalan numbers?
\end{quest}

We will construct the \emph{spectral curve} of Catalan numbers from the $(g,n)=(0,1)$ unstable geometry. We observe first that the recursion in Theorem \ref{catalan} for the unstable range gives the well-known Catalan recursion of Example \ref{cat}, i.e., 
\be \label{cata}
C_m=\sum_{a+b=m-1}C_a \cdot C_b.
\ee
We define the generating function of Catalan numbers 
$$z(x): = \sum_{m=0}^\infty C_m \cdot x^{-(2m+1)}.$$
The Catalan recursion \eqref{cata} is equivalent to the series $x-z(x)$ being the inverse of the $z(x)$. 
We thus discover the \textit{spectral curve of the Catalan recursion}:
\be\label{spec cata} 
 z^2-z \cdot x+1 = 0.
\ee
\noindent
This is the \textit{mirror dual} of Catalan numbers.

\subsection{Genesis of Enumerative Geometry: Gromov-Witten invariants of a point}\label{GWP}
By the orbifold isomorphism \eqref{orbifold iso} we will deduce that the count of graphs is equivalent to an enumerative question on $\Mbar_{g,n}$. In the stable range we will consider first the generating function of generalized Catalan numbers, or \emph{free energies}:
$$F_{g,n}(x_1,\dots,x_n):=
\sum_{\mu_1,\dots,\mu_n>0}
\frac{C_{g,n}(\mu_1,\dots,\mu_n)}
{\mu_1\cdots\mu_n}
\prod_{i=1}^n x_i^{-\mu_i}.$$
Surprisingly, the generating function $F_{g,n}$ knows $\chi(\cM_{g,n})$ and intersection numbers of $\Mbar_{g,n}$!

We  now perform a change of coordinates. For each of the variables $x_i$ we  introduce a variable $t_i$ by 
\be \label{variable}
x_i:=2 \cdot \frac{t_i^2+1}{t_i^2-1 }, \quad
z_i:=\frac{t_i+1}{t_i-1}.
\ee
With this change of variables $F_{g,n}(t_1,\ldots,t_n)$ becomes a \textit{Laurent polynomial} of degree\\
$3\cdot (2n-2+n)$ with beautiful geometric properties

\begin{enumerate}
\item $F_{g,n}(1,\ldots,1)=(-1)^n\chi(\cM_{g,n})$
\item $F_{g,n}(t_1,\ldots,t_n)=F_{g,n}(\frac{1}{t_1},\ldots,\frac{1}{t_n})$
\end{enumerate}

In Section \ref{spectral curve} we will answer the following question:
\begin{quest}
Why do we have to perform the change of variables \eqref{variable} in order to see the topological information encoded by $F_{g,n}$?
\end{quest}
\noindent
The leading terms of $F_{g,n}(t_1,\ldots,t_n)$ form a homogeneous polynomial of degree \\ $3(2n-2+n)$ 
 $$F_{g,n}(t_1,\ldots, t_n)^{top}=\frac{(-1)^n}{2^{2g-2+n}} \cdot 
\sum_{\substack{d_1+\cdots+d_n\\
=3g-3+n}} \la \tau_{d_1}\cdots \tau_{d_n} \ra \cdot
\prod_{i=1}^n (|2d_i-1|)!! \cdot
\left(\frac{t_i}{2}\right)^{2d_i+1},$$
where the symbol
$${\la \tau_{d_1}\cdots \tau_{d_n} \ra}
:=\int_{\Mbar_{g,n}}c_1(\bL_1)^{d_1}
\cdots c_1(\bL_n)^{d_n}$$
denotes cotangent class intersection numbers on $\Mbar_{g,n}$. Furthermore, the Catalan recursion obtained via Edge Contraction Operations in Theorem \ref{catalan} translates into an infinite system of differential equations known as the \textit{Dijkgraaf-Verlinde-Verlinde equation} \cite{DVV}
of the intersection numbers. This implies the celebrated theorem of Kontsevich, Mirzakhani, Okounkov-Pandharipande and others, on the Witten conjecture.
\begin{thm}[Theorem 6.1, \cite{OD6}]\label{DVV}
The Witten-Kontsevich intersection numbers satisfy the DVV equation. 
\end{thm}
\noindent
Surprisingly, the recursion relations of Theorem \ref{DVV} can be encoded compactly into a 
single ordinary differential equation. Namely, there exist a differential operator, what we call a \textit{quantum curve}, that annihilates the generating function of the free energies. More precisely, the first quantization result was proved by Mulase and Su\l{}kowski following a conjecture of Gukov and Su\l{}kowski:

\begin{thm} [\cite{MS}]\label{thm MS} Let $\hbar$
be a formal parameter. Then we have
$$\left(\hbar^2 \cdot \frac{d^2}{dx^2}+\hbar \cdot x \cdot \frac{d}{dx}+1\right) 
\exp\left(\sum_{2g-2+n\ge -1}\frac{1}{n!}\cdot\hbar^{2g-2+n}\cdot F_{g,n}(t,\dots,t)
\right)=0.$$
\end{thm}

Here, the scalar $\hbar$ is a deformation parameter. In Section \ref{section 5}, we will see that mathematically $\hbar$ is an extension class of line bundles on an algebraic curve.

\begin{quest} 
\begin{enumerate}
\item Why does this complicated function satisfy such a simple differential equation? 
\item Where does this differential equation come from?
\end{enumerate}
\end{quest}

Let us replace $z$ by $-\hbar \frac{d}{dx}$ in the \textit{the spectral curve equation of the Catalan recursion} \eqref{spec cata}
$$z^2-x\cdot z+1=0.$$
 We obtain precisely the differential operator known as the \textit{the quantum curve of the Catalan spectral curve} in Theorem \ref{thm MS}:
$$\hbar^2 \cdot \frac{d^2}{dx^2}+\hbar\cdot x\cdot \frac{d}{dx}+1.$$

The following questions are natural.

\begin{quest}
\begin{enumerate}
\item From the shape of the above equations, it looks like a canonical quantization of the spectral curve. Is it really the case?
\item If so, then what is the mathematical framework that explains this surprising phenomenon? 
\end{enumerate}
\end{quest}

\subsection{Cut-and-Join equation for orbifold Hurwitz numbers} Another example of the use of edge contraction operations on cell graphs gives a surprising enumerative problem of Hurwitz numbers. Let $H_{g,n}^r (\mu_1,\dots,\mu_n)$ denote the number of 
topological types 
of regular maps from a smooth curve of genus $g$ 
to $\bP^1$
with profile $(\overset{m}{\overbrace{r,\dots,r})}$ over $0\in\bP^1$, labeled profile $(\mu_1,\dots,\mu_n)$ over $\infty\in\bP^1$, and simple ramification at any other ramification points, weighted with automorphisms of such maps. These numbers are referred to as \emph{orbifold Hurwitz numbers}. For $r=1$, they count  simple Hurwitz numbers. 
\noindent
In \cite{OD16} we generalized the notion of \textit{branching graph} of Okounkov and Pandharipande 
\cite{OP} ($r=1$) to \textit{r-Hurwitz graphs} (arbitrary $r$), via combinatorics. An \textit{r-Hurwitz graphs} is a cell graph with a collection of dots associated to each vertex ([Definition 6.1, \cite{OD16}]). Counting the number of \textit{r-Hurwitz graphs} via \textit{edge contraction operations} we recover the Cut-and-Join equation for orbifold Hurwitz numbers. Denoting $\cH_{g,n}^r(\mu_1\dots,\mu_n)=\mu_1\ldots \mu_nH_{g,n}^r(\mu_1\dots,\mu_n)$, we obtain

\begin{thm}[Theorem 6.6, \cite{OD16}, Cut-and-Join equation for orbifold Hurwitz numbers]
\label{thm:ECF}
The number of arrowed Hurwitz graphs
satisfy the following edge-contraction formula.
\be
\label{ECF}
\begin{aligned}
 &\left(2g-2+\frac{d}{r}+ n\right)
 \cH_{g,n}^r(\mu_1\dots,\mu_n) 
\\
&\qquad
=
\sum_{i< j}\mu_i\mu_j
\cH_{g,n-1}^r (\mu_1,\dots,\mu_{i-1},
\mu_i+\mu_j,\mu_{i+1},\dots,
\widehat{\mu_j},\dots,\mu_n)
\\
\qquad
&+\half \sum_{i=1}^n
\mu_i \sum_{\substack{\a+\b=\mu_i\\
\a, \b\ge 1}}
\left[\cH_{g-1,n+1}^r (\a,\b,\mu_1,\dots,
\widehat{\mu_i},\dots,\mu_n)
\phantom{\sum_{\substack{g_1+g_2=g\\
I\sqcup J = \{1,\dots,\hat{i},\dots,n\}}}}
\right.
\\
&\qquad
\left.
+\sum_{\substack{g_1+g_2=g\\
I\sqcup J = \{1,\dots,\hat{i},\dots,n\}
}}
\cH_{g_1,|I|+1}^r (\a,\mu_I)
\cH_{g_2,|J|+1}^r (\b,\mu_J)
\right].
\end{aligned}
\ee
Here, $\widehat{\;\;}$ indicates the omission
of the index, and $\mu_I = (\mu_i)_{i\in I}$
for any subset $I\subset\{1,2,\dots,n\}$.
\end{thm}

The restriction to the $(0,1)$ unstable cases of Theorem \ref{thm:ECF} recovers the \textit{spectral curve of the orbifold Hurwitz numbers}, that is known as the r-Lambert curve. The computations are similar to the ones explained in Section \ref{recursion}, leading to the \textit{mirror curve of Hurwitz numbers}
$$x^r=ye^{-ry}.$$

Edge contraction operations are graphical manifestation of a Frobenius algebra structure and it was shown in [Corollary 4.8, \cite{OD16})] that they give alternative axiomatic definition of 2 dimensional \textit{topological quantum field theory} (2d TQFT). We further emphasize the importance of these operations by relating the 2d TQFT formula of [Corollary 4.8, \cite{OD16})] with the count of points of a character variety for a finite group and Hodge-Deligne polynomial of a character variety in \cite{OD21, OD22}.

While Catalan numbers have an algebraic spectral curve, Hurwitz numbers have an analytic spectral curve. Later on we will focus on rank $2$ Higgs bundles, whose spectral curves are algebraic. In Section \ref{spectral curve}, we will encounter
  with the familiar Catalan example.


\section{A walk into the woods of Higgs bundles and connections}\label{section 3}

Let $C$ be a smooth projective curve, 
and $K_C$ the canonical bundle of $C$ whose sections are holomorphic $1$-forms.

\subsection{Moduli spaces of vector bundles}\label{vb} We recall  the somewhat anachronistic 
definition of a holomorphic  vector bundle $E$ over $C$. For an open cover of affine sets $C=\cup_{\a} U_{\a}$, we denote by $f_{\a\b}:U_{\a}\cap U_{\b}\stackrel{}{\rightarrow} GL_{r}(\mathbb{C})$ the holomorphic transition functions that satisfy  the 1-cocyle condition $f_{\a\b}=f_{\a\gamma}\cdot f_{\gamma\b}$ on $U_{\a}\cap U_{\b}\cap U_{\gamma}$. Two points
 $(x_{\a}, \zeta_{\a}(x_{\a}))\in U_{\a}\times \mathbb{C}^r$ and
 $(x_{\b}, \zeta_{\b}(x_{\b}))\in U_{\b}\times \mathbb{C}^r$ are glued
 if $\zeta_{\a}(x)=f_{\a\b}(x)\cdot \zeta_{\b}(x)$ for $x\in U_{\a}\cap U_{\b}$. Two transition functions $f_{\a\b}$ and $f'_{\a\b}$ subordinating  the same open covering $U_{\a}$ define the isomorphic
 vector bundle if and only if there exists a family $u_{\a}:U_{\a}\stackrel{}{\rightarrow} GL_{r}(\mathbb{C})$ of holomorphic maps, called
 \emph{gauge transformations}, satisfying
$$f_{\a\b}'=u_{\a}\cdot f_{\a\b} \cdot u_{\b}.$$ A global holomorphic section of the vector bundle $E$, $s\in H^0(C,E)$, is given by a collection of holomorphic maps $s_{\alpha}:U_{\a}\stackrel{}{\rightarrow} \mathbb{C}^r$ compatible with the transition functions: $s_{\a}(x)=f_{\a\b}(x) \cdot s_{\b}(x)$ for $x\in U_{\a}\cap U_{\b}$. 

The $degree$ of a vector bundle is the first Chern class, $\deg(E):=c_1(\Lambda^r(E))$. Over a compact connected Riemann surface $C$, topologically vector bundles are completely classified by the discrete invariants, rank and degree. However, introducing a holomorphic structure, classification results of vector bundles on a smooth, irreducible, complex projective curve become more elaborate. Define the slope of $E$ to be the rational number $\mu(E)=\deg(E)/\rank(E)$ - this is a topological quantity with important implications on holomorphic structures. A holomorphic vector bundle $E$ is called \textit{stable} (resp. \emph{semi-stable}) if for any non-trivial holomorphic subbundle $F$,  $\mu(F)<\mu(E)$ (resp. $\mu(F)\leq \mu(E)$) holds.
There are complete classification results for holomorphic vector bundles for rational and elliptic curves, due to Grothendieck for the case of the Riemann sphere \cite{Gr}, and Atiyah for the case of elliptic curves \cite{At}. However, for genus higher than one there are no such classification results available, therefore such question is replaced by the construction of \textit{moduli space of 
stable holomorphic vector bundles} of rank $r$ and degree $e$, denoted by $\mathcal{U}_C(r,e)$, whose geometry has been  intensely studied. Over a smooth projective curve $C$ of genus $g>1$, the moduli space 
$\mathcal{U}_C(r,e)$ is a quasi-projective complex variety of dimension $r^2(g-1)+1$ (Narasimhan-Seshadri \cite{NS},   Seshadri \cite{Se}, see also 
Atiyah-Bott \cite{AB} and
Mumford-Fogarty-Kirwan \cite{MFK} for more information on the moduli theory of 
stable vector bundles over Riemann surfaces). 

By reducing the Yang-Mills self-duality equations from dimension 4 to dimension 2 on a compact Riemann surface, Hitchin introduced in \cite{H1} \textit{the moduli space of solutions to Hitchin's equations}, having a $\mathbb{CP}^1$ of complex structures, parametrized by $\zeta\in \mathbb{CP}^1$. If $\zeta$ is zero, then this space is identified with the \textit{Dolbeault moduli space $\mathcal{M}_{\Dol}$ of holomorphic stable Higgs bundles} 
consisting of $(E, \phi)$, where $E$ is a vector bundle of rank $r$ and fixed degree $e$, 
and $\phi\in H^0(C,\End  (E)\tensor K_C)$ is 
a \emph{Higgs field}.
 If $r$ and $e$ are coprime, then $\mathcal{M}_{\Dol}$ becomes quasi-projective variety. If $\zeta$ is non-zero and $e=0$, then the the moduli space of solutions to Hitchin's equations can be identified,
 as a complex analytic variety,
  with \textit{the de Rham moduli space 
  $\mathcal{M}_{\deR}$}
  consisting of irreducible flat connections $\nabla$
  in holomorphic vector bundles $V$ of rank $r$ and 
  degree $0$. 

The cotangent bundle $T^{*}\cU_{C}(r,e)$ is an open dense subset of $\mathcal{M}_{\Dol}$ whose complement has codimension $2$ or higher. Therefore, $\dim \mathcal{M}_{\Dol}= \dim\; T^{*}\cU_{C}(r,e)=2 \dim\; \cU_{C}(r,e)= 2\cdot r^2 \cdot (g-1)+2$,
and it acquires a holomorphic symplectic structure.

These two moduli spaces, $\mathcal{M}_{\Dol}$ and $\mathcal{M}_{\deR}$, will play a key role in our discussion.

\begin{Def}\label{Higgs def}
Let $C$ be a smooth projective curve of genus at least two, $E$ and $V$  two holomorphic rank r vector bundles on $C$ and $d$ the exterior differential on $C$.
\begin{enumerate} 
\item A \textit{holomorphic Higgs bundle} is a pair $(E, \phi)$, where $E$ is a holomorphic vector bundle and $\phi:E\stackrel{}{\rightarrow}E\otimes K_C$ is a $\mathcal{O_C}$-module homomorphism, i.e., $\phi(f \cdot s)=f \cdot \phi(s)$, $\forall f\in \mathcal{O_C}, s\in E$.
\item A \textit{stable}  (resp. \emph{semi-stable})
Higgs bundle is a Higgs pair $(E, \phi)$ such that for any $\phi$-invariant sub-bundle $F$ of $E$, $\phi:F\stackrel{}{\rightarrow}F\otimes K_C$,  $\mu(F)<\mu(E)$ (resp. $\mu(F)\leq \mu(E)$) holds.
\item A  \textit{holomorphic connection} is a 
$\mathbb{C}$-linear homomorphism  $\nabla: V\stackrel{}{\rightarrow}V\otimes K_C$ 
of a holomorphic vector bundle $V$  such that 
$\nabla(f \cdot s)=df\tensor  s+f\cdot \nabla(s)$, $\forall f\in \mathcal{O_C}$ and holomorphic 
sections $s\in V$. A \emph{differentiable connection}
is defined in the same way, replacing $f$
by a differentiable function on $C$ and $s\in V$ 
by a differentiable section of $V$.
\item An \textit{irreducible connection} is a connection $\nabla$ in $V$ for which no sub-bundle of $V$  is $\nabla$-invariant.
\item A \textit{hermitian metric} on the complex vector bundle $E$
is a positive definite Hermitian form $h$
on each fiber $E_p$, $p\in C$. It is a smooth section of $\Gamma(E\otimes \bar{E}^*)$ such that for all $\eta, \zeta\in E_{p}$,
$$\la \eta,\zeta\ra:= h_{p}(\eta, \bar{\zeta})=h_{p}(\zeta, \bar{\eta}) \text{ and } h_p(\zeta, \bar{\zeta})>0.$$
\item Let $h$ be a hermitian 
metric in a vector bundle $V$, and 
$\la\cdot,\cdot\ra$ the hermitian inner 
product. An $h$-\textit{unitary connection} on a 
 vector bundle $V$ is a 
 differentiable 
connection $\nabla$ such that for any 
differentiable sections $s$ and $t$ of $E$,   $\la s,t\ra=\la \nabla(s),\nabla(t)\ra$.
\end{enumerate}
\end{Def}

\begin{rem}
If $\nabla_1$ and $\nabla_2$ are 
holomorphic connections in a 
holomorphic vector bundle $V$, then the difference
$\nabla_1-\nabla_2$ is an $\mathcal{O}_C$-module
homomorphism.  Therefore, $(V, \nabla_1-\nabla_2)$ is a holomorphic
Higgs pair. This proves that the Dolbeault and the de Rham moduli spaces, $\mathcal{M}_{\Dol}$ and $\mathcal{M}_{\deR}$, have the same dimensions.
\end{rem} 

\subsection {Spectral curve of the Hitchin fibration} \label{}
The Higgs field $\phi:E\stackrel{}{\rightarrow}E\otimes K_C$ induces a map $\wedge^i \phi:\wedge^i E\stackrel{}{\rightarrow}\wedge^i E\otimes K_C^i$
for every $i\ge 0$, locally given by the alternative sum of the $i$-th minors of $\phi$, with trace $\tr(\wedge^i \phi)\in H^0(C, K_C^i)$. If $i=r$ then $\wedge^r \phi =\tr (\wedge^r \phi)=\det \phi \in H^0(C, \End (\wedge^r E)\otimes K_C^r)$. This defines a holomorphic map, called the \textit{Hitchin map} $H$, which 
induces an algebraically
completely integrable Hamiltonian system
\be
\label{Hitchin map}
\begin{CD}
\mathcal{M}_{\Dol}
\\
@VV{H}V
\\
\;\;\;B:=\bigoplus_{i=1}^{r}H^0(C, K_{C}^i).
\end{CD}
\ee
For a Higgs pair $(E,\phi)$ let $s$ denote the \textit{spectral data} 
$$s:=\big((-1)^{i}\tr(\wedge^{i}\phi)\big)_{i=1,\dots,r}\in B\cong \mathbb{C}^{r^2(g-1)+1}.$$
The Hitchin map $H$ sends a Higgs bundle  to its spectral data
$$(E,\phi)\stackrel{H}{\rightarrow}s.$$

Obviously, the fiber  of $H$ over zero
contains all Higgs bundles 
of the form $(E, \phi=0)$, where $E$ is a stable vector bundle, so $\mathcal{U}_C(r,e)\subset H^{-1}(0)$.

The total space of the canonical bundle $K_C$ is the cotangent bundle $T^*C$. There is a tautological $1$-form $\eta \in H^0(T^*C,\pi^*K_C)$
on the symplectic variety $T^*C$ defined by
$$
\begin{CD}
T^*C@<<<\pi^*K_C\\
@V{\pi}VV
\\
C@<<< K_C.
\end{CD}
$$
Locally $\eta$ is defined by $ydx$ for $x\in C$ and $y$ the fiber coordinate, while
$-d\eta$ is the natural holomorphic symplectic form on $T^*C$. 

The characteristic polynomial of a Higgs bundle
$(E,\phi)$  defines the \textit{spectral curve}  denoted by $\Sigma_s$ in $T^*C$ as the divisor of the zero of the following
global section
\be \label{spectral}
\det(\eta\cdot I_r-\pi^*\phi):=\sum_{i=0}^r (-1)^i \tr(\wedge^i \phi)\otimes \eta^{\otimes(r-i)}\in H^0(T^*C, \pi^*K_C^r).
\ee
The coefficients of the 
defining equation
of $\Sigma_s$ are given by the spectral data $s$ of the Higgs bundle $(E,\phi)$. Observe that $\eta$ induces a $1$-form on the spectral curve $\Sigma_s$ by pulling back via the inclusion map $i$ of $\Sigma_s$ in $T^*C$.

\begin{rem}
General Properties of $\Sigma_s$.
\begin{enumerate}
\item $\Sigma_s$ is non-singular for generic $s$.
\item
 $\Sigma_s$ is a curve inside the cotangent bundle $T^\ast C$ of genus 
$$g(\Sigma)=r^2(g-1)+1.$$
\item
The fiber over a generic point is the Jacobian:
$$H^{-1}(s)=\Jac(\Sigma_s).$$
\item
There is a degree $r$ cover
$$
\begin{CD}
\;\Sigma_s
\\
@VV{r:1}V
\\
\;C.
\end{CD}
$$

\end{enumerate}
\end{rem}

From now on we will consider the  Hitchin theory for the Lie group $G=SL_{r}(\mathbb{C})$. In other words,  $\mathcal{M}_{\Dol}$ denotes the moduli space of holomorphic Higgs bundles $(E, \phi)$ with $\tr (\phi)=0$ such that $E$ has the fixed trivial determinant. The fiber of the Hitchin map $H^{-1}(s)$ 
at a generic point 
$$s\in B:=\bigoplus_{i=2}^r H^0(C,K_C^i)
$$
becomes the Prym variety 
$$
H^{-1}(s) = \Prym(\Sigma_s\rar C):= \Ker(Nm)
$$
of the spectral covering
 $\pi: \Sigma_s\lrar C$ with the spectral data $s$, i.e., the kernel of the norm map
 $$
 Nm: \Jac(\Sigma_s)\owns \sum_{p\in \Sigma_s}
 m_p\cdot p \longmapsto \sum_{p\in \Sigma_s}
 m_p \cdot \pi(p)\in \Jac(C).
 $$ 
 The moduli space $\mathcal{M}_{\Dol}$ is a generically Abelian fibration over $B$.

\subsubsection{Rank $2$ simplification}\label{rk 2}

We will focus next on rank $2$ and degree $0$.

\begin{enumerate}

\item The $SL_{2}(\mathbb{C})$-Higgs bundle $(E, \phi)$, has $\tr(\phi)=0$ and trivial determinant\\
 $\det(E)=\wedge^2 E=\mathcal{O}_C$.
\item The Hitchin map 

\vskip-0.2in
$$
\begin{CD}
\mathcal{M}_{\Dol}
\\
@VV{H}V
\\
B:=H^0(C, K_{C}^2)\ni s
\end{CD}
$$
sends $(E,\phi)\stackrel{H}{\rightarrow}\det(\phi)=s.$

\end{enumerate}

\begin{ex}[rank two $SL_{2}(\mathbb{C})$ stable  Higgs bundles] Choose a spin structure on 
a curve $C$ of genus $g\ge 2$, i.e., a choice of the line bundle $K_C^{\frac{1}{2}}$, \eqref{spin}. For any quadratic differential $q\in H^0(C, K_{C}^2)$  on $C$, $\left(K_C^{\frac{1}{2}}\oplus K_C^{-\frac{1}{2}}, 
\begin{pmatrix}
0&q\\ 1& 0
\end{pmatrix}\right)$
 is a  stable Higgs bundle on $C$. This example will play a key role in our later analysis (see Definition \eqref{section}). 
 
 Another example is on the  trivial
   vector bundle. For 
   any non-zero holomorphic $1$-form $p\in H^0(C,K_C)$, 
    $\left(\mathcal{O}_C\oplus \mathcal{O}_C, \begin{pmatrix}
0&p\\ p& 0
\end{pmatrix}\right)$ is again a stable Higgs bundle.
\end{ex}

\begin{rem}\label{conditions}
Let $\{f_{\a\b}\}$ be a transition function for $E$. 
\begin{enumerate}
\item Locally on $C$, $\nabla|_{U}=d+ A|_{U}$, where $A:E\stackrel{}{\rightarrow}E\otimes K_C.$

\item Any Higgs field is compatible with the transition functions of $E$. Indeed, if $\phi_{\a}=\phi|_{U_{\a}},$
then 
\be \label{bundle} 
\phi_{\a}=f_{\a\b}\cdot \phi_{\b} \cdot f_{\a\b}^{-1}.
\ee
\item If $s\in E$ is a holomorphic section, then $s_{\a}(x)=f_{\a\b}(x) \cdot s_{\b}(x)$ for $x\in U_{\a}\cap U_{\b}$. We note that derivatives of sections
do not make sense as sections: 
$ds_{\a}(x)\neq f_{\a\b}(x) \cdot ds_{\b}(x)$. Nevertheless, if
$A_{\alpha}=A|_{U_{\alpha}}$, then $(d+A_{\alpha})\cdot s_{\alpha}=f_{\a\b} \cdot [(d+A_{\b})\cdot s_{\b}]$ on $U_{\a}\cap U_{\b}$. Therefore, for any connection $\nabla=\{d+A_{\a}\}_{\a}$ on $E$, the Gauge transformation holds
\be \label{gauge}
A_{\a}=f_{\a\b} \cdot A_{\b}\cdot f_{\a\b}^{-1}-f_{\a\b}^{-1}\cdot df_{\a\b}.
\ee
Since $f_{\a\b}\cdot f_{\a\b}^{-1}=1$, by applying the exterior differential $d$, it is obvious that \eqref{gauge} is equivalent to $A_{\a}=f_{\a\b} \cdot A_{\b}\cdot f_{\a\b}^{-1}+f_{\a\b}\cdot df_{\a\b}^{-1}$, or simply $df_{\a\b}=f_{\a\b}\cdot A_{\b}-A_{\a} \cdot f_{\a\b}$.

\end{enumerate}
\end{rem} 

Locally, $\phi$ and $A$ are both $r$ by $r$ matrices of 1-forms, but satisfying different rules with respect to the transition function of the vector bundle. We emphasize here that given a Higgs bundle $(E, \phi)$ in $\mathcal{M}_{\Dol}$, it is not obvious how to obtain a connection $(V, \nabla)$ in $\mathcal{M}_{\deR}$. The goal of these lecture notes is to reveal the holomorphic \textit{path} that a Higgs bundle on a Hitchin section 
travels to become a connection, called an \textit{oper}, is the goal of these lecture notes.


\section{From Higgs bundles to quantum curves}\label{section 4}

\subsection{Higgs bundles for Catalan numbers}\label{spectral curve}

We wish to consider the  curve $z^2-z\cdot x+1 = 0$ of \eqref{cata} as a local expression of a singular 
spectral curve (divisor) inside the Hirzebruch surface $\mathbb{F}_2$ associated with 
a meromorphic Higgs bundle.

\begin{ex}{The \textbf{Spectral curve} of Catalan
numbers  as the \textbf{spectral curve} of a Higgs bundle.}
\begin{enumerate}
\item The curve is $C=\mathbb{P}^1$, and the vector bundle $E=K_C^{\frac{1}{2}}\dsum K_C^{-\frac{1}{2}}=\mathcal{O}_{\bP^1}(-1)\oplus \mathcal{O}_{\bP^1}(1).$
\item The \textbf{meromorphic Higgs field} $\phi:E\stackrel{}{\rightarrow}E\otimes K_C(\ast)$ is given by
$$\phi=\begin{pmatrix}
0&-(dx)^2\\1& x\cdot dx
\end{pmatrix}$$
on the affine line $\bA^1\subset \bP^1$.
\item The \textit{spectral curve of a Higgs bundle}, denoted by 
$$\Sigma\subset \overline{T^*\mathbb{P}^1}=\mathbb{F}_2: = \bP\big(\mathcal{O}_{\bP^1}(-2)\oplus \mathcal{O}_{\bP^1}\big),
$$ is given by the characteristic polynomial of the Higgs field $\phi$ 
$$\det(\eta \cdot I_2-\phi)=\det(z\cdot dx\cdot I_2-\phi)=(z^2-x\cdot z+1)\cdot(dx)^2=0$$ in $T^*\bA^1$.
\end{enumerate}

\end{ex}

\begin{itemize}
\item We further consider a \textbf{resolution of singularity} of curve $\Sigma$ by blowing up $\bF_2$.
\item $\Sigma$ is smooth  in $T^*\bA^1$ near 
$(0,0)$ (Figure \ref{zero}),
and  has a \textit{double point} in $\mathbb{F}_2$ at $(\infty, \infty)$ (Figure \ref{infinity}). 

\begin{figure}[h]
\centering
\subfloat[$\Sigma$ around $(0,0)$: $z^2-x\cdot z+1=0$]{{\label{zero}\includegraphics[width=.45\textwidth]{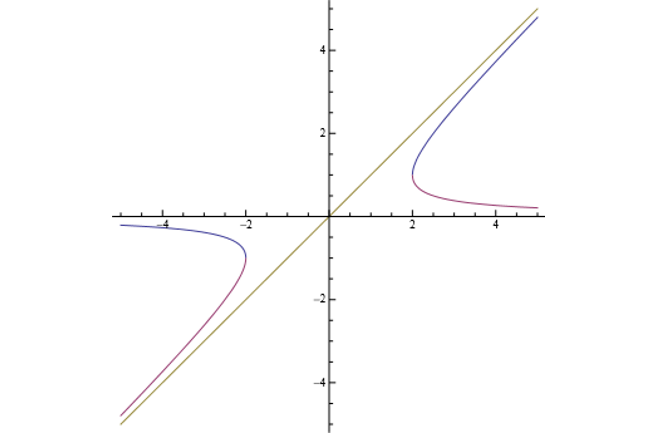}}}%
\qquad \qquad
\subfloat[$\Sigma$ at $(\infty,\infty)$: $u^4-u\cdot w+w^2=0$]{{\label{infinity}\includegraphics[width=.45\textwidth]{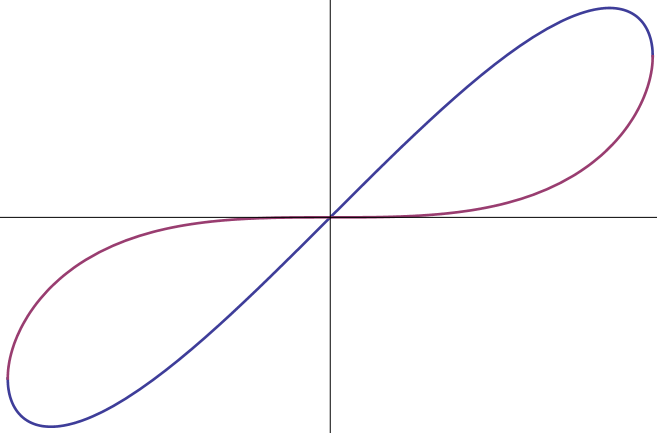}}}%
\caption{}
\end{figure}

\item Blow up the surface $\mathbb{F}_2$ and denote by $\widetilde{\Sigma}$ the strict transform of $\Sigma$. Introduce a new local parameter $w_1$ such that $w=w_1\cdot u$.  The strict transform $\widetilde{\Sigma}$ becomes a conic
$$\widetilde{\Sigma}: u^2+(w_1-\frac{1}{2})^{2}=\frac{1}{4}.$$
\item $\widetilde{\Sigma}$ is a  rational curve, and $t$ of \eqref{variable} \textbf{is the normalization coordinate} that parametrizes  the affine part of the spectral curve by 
\begin{align*}
u&=\frac{t^2-1}{2\cdot(t^2+1)} 
\\ 
w_1&=\frac{1}{2}-\frac{t}{t^2+1}.
\end{align*}
\item Denote by $F$  the class of a fiber
of $\bF_2\rar \bP^1$, by $B$ the negative section, 
i.e., the zero section of 
$T^*\bP^1$, and by $E$  the exceptional divisor created on $\mathbb{F}_2$ after the blow-up at the double point of $\Sigma$ at $(\infty,\infty)$. Then the proper transform $\widetilde{\Sigma}$ on the blown-up of $\mathbb{F}_2$ is written as the divisor 
$$\widetilde{\Sigma} =  4F+2B-2E.$$
\end{itemize}

\subsection{Higgs bundles and quantum curves}
We are now ready to generalize Theorem \ref{thm MS} and results in Section \ref{spectral curve} to any meromorphic Higgs bundle $(E, \phi)$ of rank two. 

In \cite{OD8} we have established a \textbf{new connection} between the Hitchin theory/Higgs bundles and topological recursion/quantum curve theory. These are two apparently different broad theories that share the notion of \textbf{spectral curves}. To establish the notations, let $C$ be a smooth projective curve of \textit{arbitrary genus}, and $K_C$ the canonical bundle. We denote by $E$  a holomorphic \textit{rank two} vector bundle on $C$, and by $\phi:E\stackrel{}{\rightarrow}E\otimes K_C(\ast)$ a Higgs field. 
\begin{itemize}
\item In \cite{OD8} we considered  a \textit{holomorphic} Higgs pair $(E, \phi)$. Hitchin constructed the spectral curve $\Sigma$ of $\phi$ by the characteristic polynomial of $\phi$ \eqref{spectral}, $\Sigma\hookrightarrow T^*C$.
\item In \cite{OD12} we considered   a \textit{meromorphic} Higgs pair. We construct the spectral curve  $\Sigma$ as the zero divisor of the characteristic polynomial of $\phi$ inside the compactified cotangent bundle of $C$ that is a ruled surface over $C$:   

$$\Sigma:= (\det(\eta \cdot I_r-\pi^*\phi))_{0}\hookrightarrow \overline{T^*C},$$
\noindent
where $\eta \in H^0(T^*C,\pi^*K_C)$ is the tautological 1-form on $T^*C$ extended as a meromorphic 1-form on the compactification $\overline{T^*C}$. We consider a \textbf{resolution of singularities} of $\Sigma$ by blowing up the ruled surface $\overline{T^*C}$ over $C$,  along the base locus of $\Sigma$.
$$
\begin{CD}
\widetilde{\Sigma}@>{i}>>Bl (\overline{T^*C}) 
\\
@V{}VV @VV{blow-up}V 
\\
\Sigma@>>{i}> \overline{T^*C} 
\end{CD}
$$
\end{itemize}

In \cite{OD8}, \cite{OD12} (see also \cite{OD17}), we extended the framework of topological recursion 
\cite{EO} to singular Hitchin spectral curves, utilizing the birational geometry of ruled surfaces. As a consequence, this extension has led to the discovery of the relation between \textbf{Hitchin spectral curves} and \textbf{Gromov-Witten invariants} in few examples (as the one in Section \ref{catalan} and Section \ref{spectral curve}). More precisely, the novelty of this approach is the discovery of the \textbf{PDE differential recursions of free energies} $F_{g,n}$ in [Definition 6.6, \cite{OD12}] (as well as [Equation 6.5, \cite{OD8}]) that implies the WKB analysis of the quantization Theorem \ref{quantum}. Moreover, the \textbf{PDE differential recursions of free energies} $F_{g,n}$ also implies the well-known integral \textbf{topological recursion} of Eynard-Orantin for a 
spectral curve of genus $0$. The PDE recursion relates the Hitchin spectral curve with enumerative geometry.

\begin{thm}[Quantization Theorem \cite{OD8}, \cite{OD12}]\label{quantum}
For a rank 2 Higgs bundle and $x\in C$, we construct locally a second order differential operator $P(x,\hbar \cdot d/dx)$ whose semi-classical limit recovers the spectral curve $\Sigma$. We also construct a solution $\psi(x, \hbar)$ of equation $P(x,\hbar \cdot d/dx)\psi(x,\hbar)=0$ in terms of principal specialization of the PDE recursion.
\end{thm}

The enumerative geometry example of the Catalan numbers emphasized by equation \eqref{spec cata} is locally encaptured in the framework of Hitchin systems by Example in the Section \ref{spectral curve}. Following this approach, assume that the spectral curve of the Higgs bundle has the local expression
$$y^2-\tr \phi(x)\cdot y+\det \phi(x)=0.$$
The quantum curve associated to this spectral curve is a Rees $D$-module, locally given by the second order differential operator obtained by replacing the $y$ variable by $\hbar \frac{d}{dx}$ (as in Theorem \ref{thm MS})
$$P(x,\hbar \cdot d/dx)=\left(\hbar\cdot \frac{d}{dx}\right)^2-\tr \phi(x) \cdot \hbar \cdot\frac{d}{dx}+\det \phi(x).$$
The generating function of free energies is
$$\psi(x, \hbar)=\exp\left(\sum_{2g-2+n\ge -1}\frac{1}{n!}\cdot \hbar^{2g-2+n} \cdot F_{g,n}(x,\dots,x)
\right)=0,$$
where $F_{g,n}(x_1, \ldots, x_n)$ are the free energies defined by the PDE recursion 
of [Definition 6.6, \cite{OD12}]. If the spectral curve $\Sigma$ is a singular curve, then the differential operator $P(x,\hbar \cdot d/dx)$ has irregular singularities and $\psi$ has essential singularities. The asymptotic expansion  of $\psi$ (see e.g. [Definition 1.1, \cite{OD17}]) as in the Catalan example \eqref{spec cata} around its singularity has coefficients that encode information of Gromov-Witten invariants (Section \ref{GWP}, see also the Airy example of [Section 1, \cite{OD17}]). 
\newpage

\section{The metamorphosis of quantum curves into opers}\label{section 5}
\label{NAH}

From now on we will focus on holomorphic Higgs bundles $(E, \phi)$ on  a Riemann surface $C$ of genus at least two. 

\subsection{Projective coordinate system} 
We recall that a \textit{universal covering} is a covering space that is simply connected.
By Riemann uniformization theorem, every simply connected Riemann surface is biholomorphic to $\bP^1$, $\mathbb{C}$, or to the upper half-plane $\mathbb{H}:=\{z\in \mathbb{C}| Im (z)>0\}$  with a global coordinate $z$. Therefore for a Riemann surface of genus at least two, the universal covering is the upper half-plane.

Notice that the global coordinate on $\mathbb{H}$ induces, by the quotient map $\pi: \mathbb{H}\stackrel{}{\rightarrow} C$, a particular coordinate system on the Riemann surface $C$. Indeed, there is a faithful representation 
$$
\rho:\pi_1(C)\lrar SL(2,\bR)
$$
such that $C \isom \bH\big/\rho\big(\pi_1(C)\big)$,
where $SL(2,\bR)$ acts on $\bH$ through
the projection
$$
0\lrar \bZ/2\bZ \lrar SL(2,\bR) \lrar
PSL(2,\bR) = \Aut(\bH) \lrar 0.
$$
We can give a particular coordinate system on $C$
using the universal covering map $\pi:\bH \lrar C$.
Let 
$$
C = \bigcup_\a U_\a
$$
be an open finite cover of $C$.
For each coordinate neighborhood $U_\a$, 
choose a contractible open subset $\widetilde{U}_\a
\subset\bH$ for which the map
$$
\pi:\widetilde{U}_\a\overset{\sim}{\lrar}U_\a\subset C
$$
is a biholomorphic map. Let us denote by $z_\a$
the local coordinate defined on $U_\a$ that corresponds
to the global coordinate $z$ restricted on 
$\widetilde{U}_\a$. Then on each $U_\a \cap U_\b$,
we have a \emph{M\"obius} coordinate transformation
\be
\label{Mobius}
z_\a = \frac{a_{\a\b}\cdot z_\b + b_{\a\b}}
{c_{\a\b} \cdot z_\b + d_{\a\b}}, 
\qquad 
\begin{bmatrix}
a_{\a\b}&b_{\a\b}\\
c_{\a\b}&d_{\a\b}
\end{bmatrix}
\in SL(2,\bR).
\ee
In what follows, we choose and fix a M\"obius 
coordinate system on $C$.

Since 
$$
dz_\a = \frac{1}{(c_{\a\b} \cdot z_\b + d_{\a\b})^2}
\; \cdot dz_\b,
$$
the transition function for the canonical line bundle
$K_C$ of $C$ is given by
the cocycle 
$$
\left\{\xi_{\a\b}=\frac{dz_{\b}}{dz_{\a}}=(c_{\a\b} \cdot z_\b + d_{\a\b})^2\right\}
\quad \text{on} \quad U_\a\cap U_\b.
$$
\begin{enumerate}
\item We choose and fix a theta characteristic,
or \textbf{a spin structure for $C$}, i.e. a line bundle $K_C^{\half}$ such that 
$(K_C^{\half})^{\tensor 2} \isom K_C$.
\item Let $\{\xi_{\a\b}\}$ denote the $1$-cocycle 
corresponding to $K_C^{\half}$
with respect to the M\"obius coordinate system. 
\item The transition functions for $K_C^{\frac{1}{2}}$ are given by
\be 
\label{spin}
\xi_{\a\b} = \pm (c_{\a\b}\cdot z_\b + d_{\a\b}).
\ee
\end{enumerate}
The choice of the $\pm$ sign here is exactly 
an element of $H^1(C,\bZ/2\bZ) = (\bZ/2\bZ)^{2g}$,
which classifies the spin structure of $C$.

\begin{Def}[Gunning 1967 \cite{Gun}]\label{gun}
A \textit{projective coordinate system} on $C$ is a coordinate system on which transition function is given by a \emph{M\"obius} transformation.
\end{Def}
$$
C = \bigcup_\a U_\a, \hskip .1in z_\a\in U_{\a}, \hskip .1in z_\a = \frac{a_{\a\b} \cdot z_\b + b_{\a\b}}
{c_{\a\b}\cdot z_\b + d_{\a\b}}, 
 \hskip .1in
\begin{bmatrix}
a_{\a\b}&b_{\a\b}\\
c_{\a\b}&d_{\a\b}
\end{bmatrix}
\in SL(2,\bC).
$$ 

\subsection{Hitchin section in rank two}\label{HC rank 2}

Equipped with the choice of a spin structure for $C$ and   the transition functions
$\xi_{\a\b}$ for the line bundle $K_C^\half$, we  define the \textit{Hitchin section in rank two}.

Recall the Hitchin map for the $SL_2(\mathbb{C})$-Higgs bundles sends $\cM_{\Dol}\owns
(E,\phi)\stackrel{H}{\rightarrow}\det(\phi)\in B$. 
\begin{Def}\label{section}
For each choice of $q\in  H^0(C, K_C^2)=B$, the \textbf{Hitchin section} is the holomorphic Lagrangian inside the Dolbeault moduli space $\mathcal{M}_{\Dol}$, given by 
$$s(q)=\left(E_0:=K_C^{\frac{1}{2}}\oplus K_C^{-\frac{1}{2}}, \phi(q):=\begin{bmatrix}
0&q\\
1&0
\end{bmatrix}\right)
.$$

\end{Def}

\begin{figure}[htb]
\includegraphics[width=2.5in]{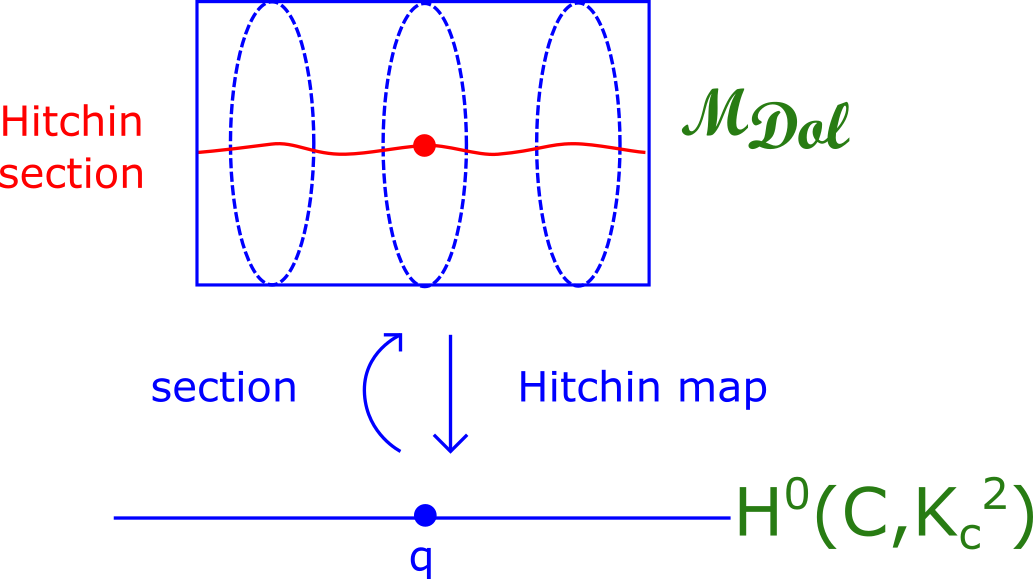}
\end{figure}

Let $f_{\a\b}=\begin{pmatrix}
\xi_{\a\b}&0
\\ 0& \xi_{\a\b}^{-1}
\end{pmatrix}$ be the transition functions of the vector bundle $K_{C}^{\frac{1}{2}}\oplus K_{C}^{-\frac{1}{2}}$.
If  the quadratic differential $q$ has 
a local form $q(z)|_{U_{\a}}=q_\a(z_{\a})\cdot dz_{\a}^2$, then the Higgs field on the Hitchin section $\phi(q)$, which is a matrix valued $1$-form, is given by 
$$
\phi_{\a}=\begin{pmatrix}
0&q_\a(z_{\a})\cdot dz_{\a}\\ dz_{\a}& 0
\end{pmatrix}.
$$
Notice that the Higgs field $\phi(q)$ satisfies the compatibility condition \eqref{bundle}:
\begin{align}
f_{\a\b}\cdot\phi_{\b}\cdot f_{\a\b}^{-1}&=
\nonumber
\begin{pmatrix}
\xi_{\a\b}&0\\ 0& \xi^{-1}_{\a\b}
\end{pmatrix}\cdot
\begin{pmatrix}
0&q_\b(z_{\b})\cdot dz_{\b}\\ dz_{\b} & 0
\end{pmatrix}\cdot
\begin{pmatrix}
\xi_{\a\b}^{-1}&0\\ 0& \xi_{\a\b}
\end{pmatrix}\\
\nonumber
&=\begin{pmatrix}
0& \xi_{\a\b}^2\cdot q_\b(z_{\b})\cdot dz_{\b}\\ \xi_{\a\b}^{-2} \cdot dz_{\b}& 0
\end{pmatrix}
=\phi_{\a}.
\end{align}
It follows from noticing
$q_\b(z_{\b})dz_{\b}^2=q_\a(z_{\a})dz_{\a}^2$ and
$\xi_{\a\b}^{2}=\frac{dz_{\b}}{dz_{\a}}$, concluding that
$$\xi_{\a\b}^{2}\cdot q_\b(z_{\b})\cdot dz_{\b}=q_\b(z_{\b})\cdot \frac{dz_{\b}^2}{dz_{\a}}=
q_\a(z_{\a})\cdot dz_{\a}.$$

The stability of  the Higgs bundle $s(q)$ 
with 
$$\phi(q)=\begin{bmatrix}
0&q\\
1&0
\end{bmatrix}:K_C^{\frac{1}{2}}\oplus K_C^{-\frac{1}{2}}\stackrel{}{\rightarrow}
K_C^{\frac{3}{2}}\oplus K_C^{\frac{1}{2}}
$$
is seen as follows.
First observe that if $q\neq 0$, then
any vector sub-bundle of $E_0$, either
$0\oplus K_C^{\frac{1}{2}}$ or $K_C^{-\frac{1}{2}}\oplus 0$, is not $\phi$ invariant. If $q=0$, then the vector sub-bundle  $0\oplus K_C^{-\frac{1}{2}}$ is $\phi$-invariant since it is mapped to zero by $\phi$. However, its slope $1-g$ is negative, since we assume
$g\ge 2$.

\begin{rem}
The Hitchin section (sometimes called the Hitchin component) is a \emph{section} of the Hitchin fibration \ref{rk 2} in the
sense that it intersects with  each fiber of $H$ 
exactly once. 
For the case of $SL_2(\bC)$-Higgs bundles (rank 2),
it is also a section with respect to the Hitchin map $H$, since $H\circ s= Id_B$. However,
in general, the Hitchin sections we construct
in Section 6 are not the section with respect to
the Hitchin map $H$, because $H\circ s \ne Id_B$.
\end{rem}

For $q\in  H^0(C, K_C^2)$, the differential operator for a Higgs pair on a Hitchin section \eqref{section} in Theorem \ref{quantum}, i.e., the \textit{quantum curve}  $P(x,\hbar \cdot d/dx)|_{\hbar=1}=\frac{d^2}{dx^2}-q(x)$, is not globally defined. This is because
unlike the exterior differentiation $d$
which is globally defined,
the second order differentiation $d^2/dx^2$
has no global meaning. However, in the \textit{projective coordinate system}, the differential equation $\left(\frac{d^2}{dx^2}-q(x)\right)\psi(x)=0$ 
makes sense globally on the curve $C$,
provided that $\psi$ is a (multi-valued) section 
of $K_C^{-\half}$. 
More precisely, with respect to a coordinate change  $x=x(u)$, we have 
$$
\psi\big(x(u)\big) \frac{1}{\sqrt{dx}}
= \psi(u)\frac{1}{\sqrt{du}}
\quad
\Longleftrightarrow
\quad
e^{-\half \log x'(u)} \psi\big(x(u)\big)
= \psi(u),
$$
where $x'(u) = \frac{dx}{du}$, 
and $q(u)du^2 = q(x)dx^2$.
Then
\begin{align*}
0&=du^2 \cdot e^{\half \log x'(u)}\left[\left(
\frac{d}{du}\right)^2-q(u)^2\right]\psi(u)
\\
&=
du^2 \cdot e^{\half \log x'(u)}\left[\left(
\frac{d}{du}\right)^2-q(u)^2\right]
e^{-\half \log x'(u)}\psi\big(x(u)
\big)
\\
&=
du^2\cdot\left(e^{\half \log x'(u)}
\frac{d}{du}
e^{-\half \log x'(u)}\right)^2
\psi\big(x(u)\big)
-dx^2 q(x)\psi(x)
\\
&=
du^2\cdot \left(\frac{d}{du}-\half \frac{x''}{x'}
\right)^2 
\psi\big(x(u)\big)
-dx^2 q(x)\psi(x)
\\
&=
du^2\cdot \left[\left(\frac{d}{du}\right)^2-
 \frac{x''}{x'}\frac{d}{du}-\half
 \left(\left(\frac{x''}{x'}\right)'-
 \half \left(\frac{x''}{x'}\right)^2\right)\right]
\psi\big(x(u)\big)
-dx^2 q(x)\psi(x)
\\
&=
du^2\cdot \left[\left(\psi_x(x)x'\right)_u-\frac{x''}{x'}
\psi_x(x)x'\right]
-dx^2 q(x)\psi(x)
\\
&=
\psi_{xx}(x)\left(\frac{dx}{du}\right)^2du^2
+du^2\cdot \left(\psi_x(x)x''-x''\psi_x(x)\right)
-dx^2 q(x)\psi(x)
\\
&=dx^2\cdot \left[\left(\frac{d}{dx}\right)^2-q(x)
\right]\psi(x).
\end{align*}
Here, we have used the fact that the 
Schwarzian derivative 
$$
s_u(x):=\left(\frac{x''(u)}{x'(u)}\right)'-\frac{1}{2}\left(
\frac{x''(u)}{x'(u)}\right)^2
$$
is identically $0$ if $x=x(u)= \frac{a_{\a\b} \cdot u + b_{\a\b}}
{c_{\a\b}\cdot u + d_{\a\b}}, 
 \hskip .1in
\begin{bmatrix}
a_{\a\b}&b_{\a\b}\\
c_{\a\b}&d_{\a\b}
\end{bmatrix}
\in SL(2,\bC).
$ is a M\"obius 
transformation of Definition \eqref{gun}. 
We thus conclude that the quantum curve $P(x,\hbar \cdot d/dx)|_{\hbar=1}$ in Theorem \ref{quantum} is globally defined as a \emph{twisted} $D$-module
acting on the sheaf $K_C^{-\half}$. We leave to the interested reader to check the details of this computation.
The details of this consideration will be provided in \cite{OD21,OD22}. 

Let us now present  an \textit{intuitive definition} of opers.

\begin{Def}\label{oper}
An \textit{oper} on an algebraic curve
$C$ is a \textit{globally defined} differential operator
of order $r$ acting on $K_C^{-\frac{r-1}{2}}$.
\end{Def}

\begin{rem}\textbf{Importance of Gunning's definition.} 
In a projective coordinate system of $C$, the quantum curve in Theorem \ref{quantum} $P(x,\hbar \cdot d/dx)|_{\hbar=1}$ is an oper! 
\end{rem}

\subsection{A family of 
Deligne's $\hbar$-connections.}\label{Deligne}
For a Higgs bundle in Definition \ref{section},
we interpret
$P(x,\hbar \cdot d/dx)\psi=0$ as $\nabla^{\hbar}\begin{bmatrix}
-\hbar \psi'\\
\psi
\end{bmatrix}=0.$ Indeed, the quantum curve  of the Higgs field 
$$\left( E_0=K_C^{\frac{1}{2}} \oplus K_C^{-\frac{1}{2}}, \phi=\begin{pmatrix}
0&q(x)\cdot dx\\ dx & 0 
\end{pmatrix} 
\right)$$
 is $\left(\hbar^2 \frac{d^2}{dx^2}-q(x)\right)\cdot \psi(x,\hbar) = 0$. This second order differential equation corresponds to the linear system of ODE, $\left(\hbar\cdot \nabla^{\hbar}\right) \begin{bmatrix}
 -\hbar \cdot \psi\\
\psi
\end{bmatrix}=0$,
where $\nabla^{\hbar}$ is an $\hbar$-deformation 
family of opers
\be \label{oper}
\nabla^{\hbar}=d+\frac{1}{\hbar} \cdot \begin{pmatrix}
0&q(x)\cdot dx\\ dx& 0
\end{pmatrix}.
\ee

\begin{quest}
What is the corresponding vector bundle in which this family of connections $\nabla^{\hbar}$ is defined as in Remark \ref{conditions},  \eqref{gauge}?
\end{quest}

To answer this question we  interpret the complex number $\hbar$ of Theorem \ref{quantum} as an extension class of line bundles $\hbar\in \mathbb{C}=Ext^1(K_C^{-\frac{1}{2}}, K_C^{\frac{1}{2}})\isom H^1(C, K_C)\isom H^0(C, \mathcal{O}).$

\begin{thm}[\cite{Gun}]\label{gun1}
For every $\hbar\in \mathbb{C}$, there exists a unique 
extension 
\be \label{ses}
0\stackrel{}{\rightarrow}K_C^{\frac{1}{2}}\stackrel{}{\rightarrow}V_{\hbar}\stackrel{}{\rightarrow}K_C^{-\frac{1}{2}}\stackrel{\rightarrow}{}0
\ee
such that 
\begin{enumerate}
\item  the rank $2$
vector bundle $V_{\hbar}$ is given by transition functions $\{g_{\a\b}^{\hbar}\}$, $g_{\a\b}^{\hbar}:=\begin{pmatrix}
\xi_{\a\b}&\hbar\cdot\frac{d\xi_{\a\b}}{dz_{\b}}
\\ 0& \xi_{\a\b}^{-1}
\end{pmatrix}$,

\item $V_0\cong K_C^{\frac{1}{2}}\oplus K_C^{-\frac{1}{2}}$, and
\item for $\hbar\neq 0$, all the vector bundles $V_{\hbar}$  are isomorphic.

\end{enumerate}
\end{thm}

We denote by $V:=V_{\hbar}|_{\hbar=1}$ the unique non-trivial extension of $K_C^{-\frac{1}{2}}$
by $K_C^{\frac{1}{2}}$. We will give a more detail of higher-rank cases in \cite{OD21, OD22}.

\begin{proof} We recall that $\xi_{\a\b}$ of \eqref{spin} are  transition functions of $K_C^{\frac{1}{2}}$.

\begin{enumerate}
\item It is an easy computation to check that $g^{\hbar}_{\a\b}$ are transition functions of $V_{\hbar}$ satisfying the $1$-cocycle condition. Denote by $\sigma_{\a\b}=\frac{d\xi_{\a\b}}{dz_{\b}}$; according to \eqref{spin}, it is a constant.
First, we have
\be \label{ie}
g^{\hbar}_{\a\b}\cdot g^{\hbar}_{\b\gamma}=\begin{pmatrix}
\xi_{\a\b}&\hbar\cdot \sigma_{\a\b}
\\ 0& \xi_{\a\b}^{-1}
\end{pmatrix} \cdot 
\begin{pmatrix}
\xi_{\b\gamma}&\hbar\cdot \sigma_{\b\gamma}
\\ 0& \xi_{\b\gamma}^{-1}
\end{pmatrix}=\begin{pmatrix}
\xi_{\a\b}\cdot\xi_{\b\gamma} 
&\hbar\cdot (\xi_{\a\b} \sigma_{\b\gamma}+ \xi_{\b\gamma}^{-1} \sigma_{\a\b})
\\ 0& (\xi_{\a\b}\cdot\xi_{\b\gamma})^{-1}
\end{pmatrix}.
\ee
From 
 the $1$-cocycle condition  $\xi_{\a\b}\cdot \xi_{\b\gamma}=\xi_{\a\gamma}$ and  $\xi_{\a\b}^2=\frac{dz_{\b}}{dz_{\a}}$,  we claim
\be \label{eu}
\sigma_{\a\gamma}=\xi_{\a\b} \sigma_{\b\gamma}+ \xi_{\b\gamma}^{-1} \sigma_{\a\b}.
\ee
To see this, apply the logarithmic differentiation to $\xi_{\a\gamma}=\xi_{\a\b}\cdot \xi_{\b\gamma}$. We obtain
$$
dz_{\gamma}\cdot \xi^{-1}_{\a\gamma}\cdot \frac{d \xi_{\a\gamma}}{dz_{\gamma}}=
dz_{\b}\cdot \xi^{-1}_{\a\b}\cdot\frac{d \xi_{\a\b}}{dz_{\b}}+dz_{\gamma}\cdot \xi^{-1}_{\b\gamma}\cdot\frac{d \xi_{\b\gamma}}{dz_{\gamma}}.$$
Hence
\begin{align*}
\xi^{-1}_{\a\gamma}\cdot \frac{d \xi_{\a\gamma}}{dz_{\gamma}}&=
\frac{dz_{\b}}{dz_{\gamma}}\cdot \xi^{-1}_{\a\b}\cdot\sigma_{\a\b}+ \xi^{-1}_{\b\gamma}\cdot\sigma_{\b\gamma},\nonumber\\
\sigma_{\a\gamma} &=
\xi_{\a\gamma}\cdot\xi_{\gamma\b}^2\cdot \xi^{-1}_{\a\b}\cdot\sigma_{\a\b}+ \xi_{\a\gamma}\cdot\xi^{-1}_{\b\gamma}\cdot\sigma_{\b\gamma}\nonumber\\
&=(\xi_{\a\b}\cdot \xi_{\b\gamma})\cdot\xi_{\b\gamma}^{-2}\cdot \xi^{-1}_{\a\b}\cdot\sigma_{\a\b}+ (\xi_{\a\b}\cdot \xi_{\b\gamma})\cdot\xi^{-1}_{\b\gamma}\cdot\sigma_{\b\gamma}\nonumber\\
 &=\xi_{\b\gamma}^{-1}\cdot\sigma_{\a\b}+ \xi_{\a\b}\cdot\sigma_{\b\gamma}.
\nonumber
\end{align*}
Therefore, $g^{\hbar}_{\a\b}\cdot g^{\hbar}_{\b\gamma}=g^{\hbar}_{\a\gamma}$ in \eqref{ie}.

\item Since the matrix $g_{\a\b}^{\hbar}|_{\hbar=0}$ is diagonal,  the vector bundle it defines splits $$V_0\cong K_C^{\frac{1}{2}}\oplus K_C^{-\frac{1}{2}}.$$

\item For every $\hbar\neq 0$, we show that
vector bundle $V_{\hbar}$ is isomorphic to $V$. 
Indeed, the transition functions $g_{\a\b}^{\hbar}$ and $g_{\a\b}^{\hbar}|_{\hbar=1}$ are compatible with the change of trivialization $u_{\a}:U_{\a}\stackrel{}{\rightarrow}GL_{r}(\mathbb{C})$ (see Section \ref{vb})
\begin{align}
\nonumber
u_{\a} \cdot g_{\a\b}^{\hbar}\cdot u_{\b}^{-1}& =\begin{pmatrix}
\sqrt{\hbar}^{-1} & 0\\ 0 & \sqrt{\hbar}
\end{pmatrix}\cdot
\begin{pmatrix}
\xi_{\a\b}&\hbar\cdot \sigma_{\a\b}\\ 0& \xi^{-1}_{\a\b}
\end{pmatrix}\cdot \begin{pmatrix}
\sqrt{\hbar} & 0\\ 0 & \sqrt{\hbar}^{-1}
\end{pmatrix}\\
&=\begin{pmatrix}
\xi_{\a\b}& \sigma_{\a\b}\\ 0& \xi^{-1}_{\a\b}
\end{pmatrix}=g_{\a\b}^{\hbar}|_{\hbar=1}.
\end{align}
\end{enumerate}
\end{proof}

\begin{lem}\label{gun2}
$\nabla^{\hbar}$ in \eqref{oper} is a connection on $V_{\hbar}$.
\end{lem}

\begin{proof}
We first recall equation \eqref{gauge} for a connection $\nabla^{\hbar}=\{d+A^{\hbar}_{\a}\}$ \eqref{oper} on a vector bundle $V_{\hbar}$ given by transition functions $g_{\a\b}^{\hbar}$, where 

\begin{align}
\nonumber
A^{\hbar}_{\a}&=\frac{1}{\hbar}\begin{pmatrix}\label{connection}
0&q_\a(z_{\a})\cdot dz_{\a}\\ dz_{\a}& 0
\end{pmatrix}.\\
A^{\hbar}_{\a}&=g_{\a\b}^{\hbar} \cdot A^{\hbar}_{\b}\cdot (g_{\a\b}^{\hbar})^{-1}+g_{\a\b}^{\hbar}\cdot d(g_{\a\b}^{\hbar})^{-1}\\
\nonumber
g_{\a\b}^{\hbar} \cdot A^{\hbar}_{\b}\cdot (g_{\a\b}^{\hbar})^{-1}&=\frac{1}{\hbar}
\begin{pmatrix}
\xi_{\a\b}&\hbar\cdot \sigma_{\a\b}\\ 0& \xi^{-1}_{\a\b}
\end{pmatrix}\cdot
\begin{pmatrix}
0&q_\b(z_{\b})\cdot dz_{\b}\\ dz_{\b}& 0
\end{pmatrix}\cdot
\begin{pmatrix}
\xi_{\a\b}^{-1}&-\hbar\cdot\sigma_{\a\b}\\ 0& \xi_{\a\b}
\end{pmatrix}.\\
\nonumber
&=\begin{pmatrix}
\sigma_{\a\b}\cdot\xi_{\a\b}^{-1}&-\hbar\cdot\sigma_{\a\b}^2+ \frac{1}{\hbar}\xi_{\a\b}^2\cdot q_\b(z_{\b})\\ \frac{1}{\hbar}\cdot\xi_{\a\b}^{-2}& -\xi_{\a\b}^{-1}\cdot \sigma_{\a\b}
\end{pmatrix}\cdot dz_{\b}.\\
\nonumber
g_{\a\b}^{\hbar}\cdot d(g_{\a\b}^{\hbar})^{-1}&=-\frac{dg_{\a\b}^{\hbar}}{dz_{\b}}\cdot(g_{\a\b}^{\hbar})^{-1}\cdot dz_{\b}\\
\nonumber
&=-\begin{pmatrix}
\sigma_{\a\b}&\hbar\cdot\frac{d\sigma_{\a\b}}{dz_{\b}}\\ 0& \frac{d\xi^{-1}_{\a\b}}{dz_{\b}}
\end{pmatrix}\cdot \begin{pmatrix}
\xi_{\a\b}^{-1}&-\hbar\cdot\sigma_{\a\b}\\ 0& \xi_{\a\b}
\end{pmatrix}\cdot dz_{\b}\\
\nonumber
&=-\begin{pmatrix}
\sigma_{\a\b}\cdot\xi_{\a\b}^{-1}&-\hbar\cdot\sigma_{\a\b}^2+\hbar\cdot \xi_{\a\b}\cdot\frac{d^2\xi_{\a\b}}{dz_{\b}^2}\\ 0& -\xi_{\a\b}^{-1}\cdot \sigma_{\a\b}
\end{pmatrix}\cdot dz_{\b}.
\nonumber
\end{align}
Relation \eqref{spin} implies $\frac{d^2\xi_{\a\b}}{dz_{b}^2}=0.$ 
We conclude that
\begin{align}
\nonumber
g_{\a\b}^{\hbar} \cdot A^{\hbar}_{\b}\cdot (g_{\a\b}^{\hbar})^{-1}+g_{\a\b}^{\hbar}\cdot d(g_{\a\b}^{\hbar})^{-1}&=\frac{1}{\hbar}\cdot
\begin{pmatrix}
0&q_\b(z_{\b})\cdot\frac{dz_{\b}^2}{dz_{a}}\\ dz_{\a}& 0
\end{pmatrix}=\frac{1}{\hbar}\cdot \begin{pmatrix}
0&q_\a(z_{\a})\cdot dz_{a}\\ dz_{\a}& 0
\end{pmatrix}=A_{\a}^{\hbar}.
\end{align}
\end{proof}

By fixing a complex structure of the curve $C$
Gunning proved the following isomorphism as affine spaces  in \cite{Gun}
\be
\begin{aligned}
H^{0}(C, K_C^{2})&\cong \text{moduli space of $SL_2(\mathbb{C})$-opers on $C$}\\
&\cong \text{moduli space of \textit{projective coordinate systems} on $C$}
\\
&\qquad 
\text{that
subordinate the complex structure of $C$}.
\end{aligned}
\ee
Since the space of quadratic differentials $H^{0}(C, K_C^{2})$ is a vector space, it 
seems to imply that the holomorphic Lagrangian of opers also inherits an origin, corresponding to $q=0$.
Indeed,
 $\nabla_{unif}=d+\begin{pmatrix}
0&0\\ dx& 0
\end{pmatrix}
$
that we call the \textit{uniformizing oper}, will 
play an important role in the next two sections. 
However, we note that such a choice does not
come from algebraic geometry, as we see below.

The computations performed in Definition \ref{section} and Lemma \ref{gun2} show that the family $\hbar \cdot \nabla^{\hbar}$, as well as the \textit{quantum curve} of Theorem \ref{quantum}, is a \textbf{$\hbar$-connection of Deligne}.  This is a family of deformations that interpolates a Higgs field $\hbar \cdot \nabla^{\hbar}|_{\hbar=0}$  and a genuine 
connection $\hbar \cdot \nabla^{\hbar}|_{\hbar=1}$. We thus conclude that the Dumitrescu-Mulase quantum curve of Theorem \ref{quantum}, ${\hbar\nabla^{\hbar},} \text{ is an } \hbar\text{-deformation family}$ of connections constructing a \textit{holomorphic passage} form a Higgs field on the Hitchin section $\hbar \cdot \nabla^{\hbar}|_{\hbar=0}$ to an oper $\hbar \cdot \nabla^{\hbar}|_{\hbar=1}$, once we
choose a M\"obius coordinate system:
$$\left(K_C^{\frac{1}{2}}\oplus K_C^{-\frac{1}{2}}, \begin{bmatrix}
0&q\\
1&0
\end{bmatrix}
\right)\stackrel{DM}{\rightarrow} \left(V, d+\begin{bmatrix}
0&q(x)dx\\
dx&0
\end{bmatrix}
\right).$$


\section{Hitchin moduli spaces for the Lie group $G = SL_r(\mathbb{C})$}\label{}
To introduce Gaiotto's correspondence, we need to consider  Hitchin moduli spaces for simple complex
Lie group $G$. In this paper, we restrict 
ourselves to the case of 
 $G = SL_r(\mathbb{C})$. An $SL_r(\mathbb{C})$-Higgs bundle is a pair $(E, \phi)$ 
 consisting of a holomorphic vector bundle
 $E$ over a smooth projective curve
 $C$ with a fixed determinant $\det (E)=\wedge^{2} E=\mathcal{O}_C$, and  a traceless Higgs field 
 $\phi$. We use the same notations  in Definition \eqref{Higgs def} in Section \ref{vb}.

\begin{itemize}
\item $E, V$ denote holomorphic vector bundles of rank $r$ and degree $0$ with trivial determinant.
\item $\phi:E\stackrel{}{\rightarrow}E\otimes K_C$ is a traceless holomorphic Higgs field.

\item $\nabla:V\stackrel{}{\rightarrow}V\otimes K_C$ is an irreducible holomorphic connection. 

\setlength{\tabcolsep}{1pt}
\renewcommand{\arraystretch}{1.5}
\end{itemize}

Let $E^{top}:=E$ denote the underlying topological structure of the rank $r$ vector bundle $E$, obtained by forgetting its complex structure. Topological complex vector bundles over a compact topological surface are classified by their degrees and ranks, while the complete topological classification of complex vector bundles over a higher dimensional smooth topological manifold is given by their Chern classes. Since $E^{top}$ has rank $r$ and degree zero, it is topologically isomorphic to the direct sum $\mathcal{O}_{C}^{\dsum r}$ of $r$ copies of the trivial line bundle  $\cO_C$ over $C$.

As mentioned earlier, 
a classical result of Narasimhan-Seshadri \cite{NS}
shows that the moduli space $\cU_C(r,d)$ of stable 
holomorphic vector bundles of rank $r$ and
degree $d$ defined on 
a smooth projective algebraic curve $C$ is 
\emph{diffeomorphic} to the space of 
 projectively flat irreducible unitary  
connections
on $C$. A connection is said to be
\emph{projectively flat} if its curvature
takes values in the center of the Lie algebra of 
the structure group of the vector bundle.
For the case of degree $0$, 
there is a one-to-one correspondence
between stable holomorphic vector bundles
and  flat irreducible unitary connections. Through the Riemann-Hilbert
correspondence, these flat irreducible connections 
correspond to representations of 
the fundamental group $\pi_1(C)$ into
the unitary group modulo conjugation
\cite{AB,MFK}. The 
 equivalence classes of representations 
form a \emph{character variety}
$$\Hom^{\text{irr}}\left(\pi_1(C),U_n\right)\big/
U_n.
$$

The work of Hitchin \cite{H1}, Donaldson \cite{D}
and Simpson \cite{S} generalizes the above result
to the moduli theory of Higgs bundles, holomorphic 
connections, and complex character varieties. 
According to this 
generalization, a stable holomorphic Higgs bundle $(E, \phi)$ of degree $0$
corresponds to  $(D, \phi, h)$ consisting of the following data that satisfy \emph{Hitchin's equations}:

\begin{itemize}
\item  $h$ is a \emph{hermitian metric} on $E^{top}$. 
\item  $D$ is a \emph{unitary connection} on $E^{top}$ with respect to the hermitian metric
$h$. The connection $D$ decomposes into the holomorphic and antiholomorphic part
$$
D=D^{1,0}+D^{0,1}.
$$
In terms of a local coordinate $z$ of $C$, 
$D$  can locally be given by $D = d+A$ with the exterior  differential $d=\partial+\bar{\partial}$, where $\partial=\frac{\partial}{\partial z}\cdot dz$ and $\bar{\partial}=\frac{\partial}{\partial \bar{z}} \cdot d{\bar{z}}$, and  an $r \times r$ skew-hermitian matrix $A$ of $1$-forms on $C$.

\item $\phi:E^{top}\lrar E^{top}\tensor \Omega_C^1$ is
a  traceless $r\times r$ matrix of differentiable
$1$-forms on $C$.
\end{itemize}
We note that the Cauchy-Riemann
part $D^{0,1}$ of the connection $D$ induces a 
holomorphic structure in $E^{top}$, which we
denote simply by $E$. The unitarity condition
 means that the connection $D$ is
determined by $D^{0,1}$.

The great discovery of Donaldson \cite{D} is that
\emph{the stability condition for a Higgs bundle
$(E,\phi)$ is the system of non-linear PDEs that
Hitchin discovered through the reduction of
$4D$ Yang-Mills self-duality equations}. 
Denote by $F_D$  the curvature of the connection $D$, 
$$F_D=\half \cdot [D,D]=[D^{1,0}, D^{0,1}],
$$
and by  $\phi^{\dagger_h}$  the adjoint of $\phi$ with respect to the hermitian metric $h$.
The following system of  non-linear  PDEs is 
known as Hitchin's equations:
\be
\begin{cases}\label{hitchin eq}
 F_D + [\phi, \phi^{\dagger_h}] = 0 \\
 D^{0,1} \phi = 0.
\end{cases}
\ee
For our purpose, it is 
 important  that Hitchin's equations \eqref{hitchin eq} are equivalent to the flatness of the family of connections
\begin{equation}\label{flat}
D(\zeta):=\frac{1}{\zeta} \cdot \phi+D+\zeta \cdot \phi^{\dagger_h}
\end{equation}
for all $\zeta\in \mathbb{C}^{*}$. 
We can see this equivalence as follows. 
A straightforward calculation shows
\begin{align*}
[D(\zeta), D(\zeta)]&=\frac{1}{\zeta^2} \cdot [\phi, \phi]+\zeta^2\cdot[\phi^{\dagger_h}, \phi^{\dagger_h}]+2 \cdot (F_{D}+[\phi, \phi^{\dagger_h}])\\
&+\frac{1}{\zeta}\cdot
\left( [\phi, D]+[D, \phi]\right)+\zeta \cdot
\left( [\phi^{\dagger_h}, D]+ [D, \phi^{\dagger_h}]\right).
\end{align*}
Clearly \eqref{hitchin eq} implies the flatness of
$D(\zeta)$, because the second equation makes
$\phi$ holomorphic with respect to the 
complex structure of $C$ and the holomorphic
structure of $E$. Conversely, from the 
flatness of $D(\zeta)$, the first equation 
of  \eqref{hitchin eq} follows from the constant 
terms with respect to $\zeta$. 
From the $1/\zeta^2$ terms, we see that $\phi$ 
contains only $dz$ or $d\bar{z}$ term,
and from the $1/\zeta$ and $\zeta$ terms we see that 
either  $\phi$ or $\phi^{\dagger_h}$ is holomorphic. 
We rename the holomorphic one $\phi$ to obtain
\eqref{hitchin eq}.

 We thus have the following correspondences
$$
\fbox{A stable Higgs bundle  $(E, \phi) \longleftrightarrow  (D, \phi, h)$  satisfying  \eqref{hitchin eq} $\longleftrightarrow  [D(\zeta), D(\zeta)]=0$  of  \eqref{flat}.}
$$

To deal with three different appearances of 
complex moduli spaces in the Hitchin theory, 
we use the 
terminology \textit{gauge theoretical moduli space},
denoted by $\mathcal{M}_{\Gauge}$,  
to describe the differential geometric
moduli space of 
solutions  $(D, \phi, h)$ satisfying Hitchin's equations \eqref{hitchin eq}. It is a hyperK\"ahler manifold with  $\mathbb{P}^1$-worth  of complex structures. 
Customary, we assign the complex structure of
$\cM_{\Dol}$,
the moduli space of stable Higgs bundles, 
to the origin of $\bP^1$, and the algebraic
structure of $\cM_{\deR}$, the moduli space of
irreducible holomorphic connections, to 
$1\in \bP^1$. They are both \emph{diffeomorphic}
to $\cM_{\Gauge}$.

A particular diffeomorphism, known as
the \emph{nonabelian Hodge correspondence},
between $\cM_{Dol}$ and $\cM_{\deR}$ is
given as follows \cite{D, H1,S}.
Firstly, we assign the flat connection 
$D(\zeta)$ to $(E,\phi)\in \cM_{\Dol}$. 
Secondly, we define a new holomorphic
vector bundle $V = (E^{top},D(\zeta=1)^{0,1})$
by using the Cauchy-Riemann part of 
the flat connection
$D(\zeta)$ at $\zeta=1$.  
With respect to this complex structure, the 
$(1,0)$-part of the connection
$\nabla := D(\zeta=1)^{1,0}$ is automatically
a holomorphic connection 
in $V$, since $D(\zeta)$ is flat.
Thus we obtain $(V,\nabla)\in \cM_{\deR}$.
$$
\fbox{$\cM_{\Dol}\owns
(E, \phi)\stackrel{NAH}\longrightarrow \left(V, \nabla:=D(1)^{1,0}\right)\in \cM_{\deR}$.}
$$

This is a generalization of the classical results of Narasimhan-Seshadri to Higgs bundles. 
A character variety also comes in to the picture,
as the \emph{Betti moduli space}
$$
\cM_{\Betti} := \Hom^{\text{irr}}\left(\pi_1(C),
SL_r(\bC)\right)\sslash SL_r(\bC).
$$
The classical unitary group is now replaced by 
a complex Lie group $G=SL_r(\bC)$. 
The complex structure of $\cM_{\Betti}$ 
comes from that of the group $SL_r(\bC)$.
The Riemann-Hilbert correspondence gives
a highly transcendental biholomorphic map
between $\cM_{\deR}$ and $\cM_{\Betti}$.
We thus have 

\medskip
\begin{tabular}{r l}
$\mathcal{M}_{\text{Dol}}$ & $=$ \text{ moduli space of stable holomorphic Higgs bundles $(E,\phi)$ on $C$
of rank $r$}\\
{$\rotatebox[origin=c]{90}{$\sim$}$} & {[Diffeomorphic NAH, Donaldson-Hitchin-Simpson]}\\
$\mathcal{M}_{\text{deR}}$ & $=$ \text{ moduli space of rank r irreducible connections  $(V,\nabla)$ on $C$ }\\
{$\rotatebox[origin=c]{90}{$\cong$}$}	 & {[Biholomorphic Riemann-Hilbert]}\\
$\mathcal{M}_{\Betti}$ & $=\; \Hom^{\text{irr}}\left(\pi_1(C), SL_r(\bC)\right)\sslash SL_r(\bC)$.
\end{tabular} 
\medskip

\subsection{Hitchin section for
$SL_r(\bC)$-Higgs bundles (principal $sl_2(\mathbb{C})$)}
\label{hitchin section}
We fix a  spin structure 
$K_C^{\half}$ on $C$ 
given by 
transition functions 
$\{\xi_{\a\b}\}$. 
To define
a Hitchin section of $G$-Higgs bundles
for a simple complex Lie group $G$, we
need the notion of Konstant's 
\emph{principal three-dimensional subgroup} (TDS)
of \cite{Kostant}.
For the case of $G=SL_r(\bC)$, it  simply comes from
the unique $r$-dimensional irreducible representation
of $SL_2(\bC)$. The Lie algebra of principal TDS
is the linear span  $\la X_+,X_-,H\ra$, 
where
\begin{itemize}
\item
$X_+:=\begin{bmatrix}
0&\sqrt{p_1}& 0 &\cdots &0\\
0&      0   & \sqrt{p_2} &\cdots &0\\
\vdots&\vdots&\vdots&\ddots&\vdots\\
0&      0   & 0           &\cdots  &\sqrt{p_{r-1}}\\
0&      0   & 0           & \cdots & 0
\end{bmatrix}$, \qquad $p_i:= i(r-i)$,
\item
$X_{-}:=X_{+}^{t},$ 
\item
$H:= [X_+,X_-]=\begin{bmatrix}
r-1&    0   &\cdots &   0 & 0\\
0&      r-3  &\cdots &  0 & 0\\
\vdots&\vdots&\ddots&\vdots&\vdots\\
0&      0              &\cdots  &-(r-3) & 0\\
0&      0             & \cdots & 0     & -(r-1)
\end{bmatrix} .$
\end{itemize}
 Define the split vector bundle $E_0:=K_C^{\frac{r-1}{2}}\oplus K_C^{\frac{r-1}{2}-1}\oplus \ldots \oplus K_C^{-\frac{r-1}{2}}$, whose
  transition function 
  is given by
   $\{\xi_{\a\b}^H=\exp(H\cdot \log\xi_{\a\b})\}$.
We note that every $q_i \in H^0\!\left(C, K_C^{i+1}\right)$
satisfies that  {$q_i|_{U_{\alpha}}=q_i|_{U_{\beta}}\cdot \xi_{\alpha \beta}^{2(i+1)}$.

Now we can generalize Definition \ref{section} of Section \ref{HC rank 2} as follows.

\begin{Def}
The \textbf{Hitchin section} is a 
holomorphic Lagrangian inside $\mathcal{M}_{\Dol}$
consisting of stable Higgs pairs 
$\left(E_0,\phi(q)\right)$ for every 
 $q=(q_1,\ldots, q_{r-1}) \in B = 
 \bigoplus_{i=1}^{r-1}H^0\!\left(C, K_C^{i+1}\right)
 $, where
 $$\phi(q):=X_{-}+\sum_{i=1}^{r-1}q_i \cdot X_{+}^i.$$
\end{Def}

\begin{figure}[htb]
\includegraphics[width=2.5in]{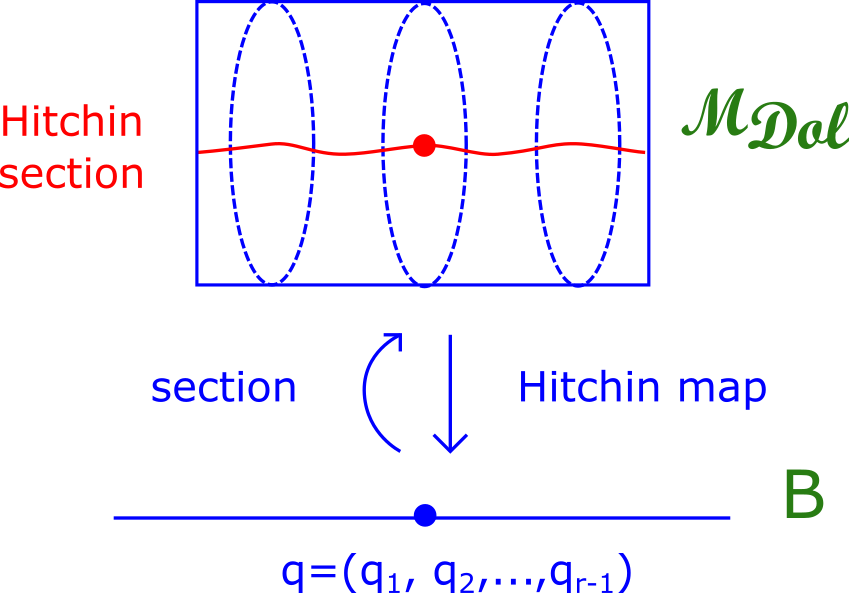}
\end{figure}

\subsection{On a conjecture of Gaiotto}\label{section 6}

In 2014, Gaiotto \cite{G} conjectured the 
following.
\begin{conj}\label{dgaiotto}
Let $(E_0,\phi)$ be a stable Higgs pair 
on a Hitchin section
in $\cM_{\Dol}$,
and $D(\zeta)$, $\zeta\in \bC^*$, the corresponding
one-parameter family of flat connections. Define 
a \textbf{two-parameter} family of connections
by 
\be
\label{two}
D(\zeta, R) := \zeta^{-1} \cdot R \cdot \phi + D + \zeta \cdot R \cdot \phi^{\dagger},
\ee
where $R\in \bR_+$ is a positive real number. 
Then the scaling limit
$$
\lim_{\substack{R\rar 0,\;\zeta\rar 0\\
\zeta/R = \hbar}} D(\zeta,R)
$$
exists, and defines an $SL_r(\bC)$-oper for
every $\hbar\in \bC^*$. 
\end{conj}

The data $(D,\phi,h)$
corresponding to any point $(E_{0}, \phi(q))$
on the Hitchin section  satisfies Hitchin's equations \eqref{hitchin eq}. Scaling the Higgs field
 $\phi(q)$ by any real parameter $R\in \mathbb{R}_{+}$ does not affect the 
 stability condition. 
 Therefore, the scaled 
 data $(D,R\phi,h)$ corresponding
 to the pair  
 $(E_{0}, R\phi(q))\in \mathcal{M}_{\Dol}$ 
 also satisfies  Hitchin's equations. This time,
 the equation is $R$-twisted:
 \be\label{R twisted}
\begin{cases}
F_D + R^2 \cdot [\phi, \phi^{\dagger_h}] = 0\\
D^{0,1} \phi= 0.
\end{cases}
\ee
By the same argument as before,
the $R$-twisted Hitchin equations are
equivalent to the flatness of the two-parameter
family of connections $D(\zeta,R)$. 

A surprising fact of this conjecture is that
the scaling limit of the differential geometric
object $D(\zeta,R)$ is automatically an $\hbar$-family
of 
holomorphic connections defined on an algebraic
$\hbar$-deformation family of filtered
vector bungles, generalizing the extension 
$$0\rightarrow K_C^{\frac{1}{2}}\rightarrow V_\hbar \rightarrow K_C^{-\frac{1}{2}}\rightarrow 0.$$
Na\"ively, it looks that the scaling limit
of \eqref{two} is simply
$D+\frac{1}{\hbar}\phi$. There is a problem here,
because the hermitian metric $h$ that solves
\eqref{R twisted} explodes as $R\rar 0$. 
Since the $h$-unitary connection $D$ 
also depends on $h$, the limit of $D$ does not
make sense as $R$ tends to $0$.

\begin{thm}[Dumitrescu, Fredrickson, Kydonakis, Mazzeo, Mulase, Neitzke, \cite{OD20}]\label{DFKMMN} 
Conjecture~\ref{dgaiotto} holds 
for an arbitrary simple and simply connected complex Lie group $G$.
\end{thm}

\subsection{Sketch of the proof in rank two} We present here the main steps to prove Theorem~\ref{DFKMMN}
for the case of $G=SL_2(\bC)$. We use the basis 
$$
\left<
X_+=\begin{bmatrix}0&1\\0&0
\end{bmatrix},
X_-=\begin{bmatrix}0&0\\1&0
\end{bmatrix},
H=\begin{bmatrix}1&0\\0&-1
\end{bmatrix}
\right>
$$
for $sl_2(\bC)$ to simplify our calculations. 
Their commutation relations are
$$
[X_+,X_-] = H,\qquad [H,X_\pm] = \pm 2X_\pm.
$$

\begin{step}
We first notice that a hermitian metric on 
the canonical bundle $K_C$ naturally comes
from a hermitian metric on the curve $C$ itself. 
Since we start from a Higgs pair on a Hitchin section with vector bundle $E_0=K_C^{\frac{1}{2}}\oplus K_C^{-\frac{1}{2}}$, a fiber metric is determined by a hermitian metric on $C$.  Recall that
$$
\fbox{A Higgs bundle $(E_0, \phi)$ on the Hitchin section   $\lrar   (D, \phi, h)$  satisfying  \eqref{hitchin eq}.}
$$
A choice of a hermitian metric 
on $C$ determines the fiber metric
$h$, and hence the unitary connection $D$. Thus
we wish to see how  it translates into
Hitchin's equations. 
\end{step}

\begin{step}
 We start from a complex structure on $C$
 with a holomorphic local parameter $z$, together
 with a  hermitian metric 
 $$g=\lambda^2 \cdot dz \cdot d\bar{z}$$
 on $C$, where
 $\lam$ is a positive
 real function depending on $R$.  The hermitian
 metric on $C$ is the same as the fiber metric of 
 the tangent bundle of $C$, which is $K_C^{-1}$. 
 Hence $\lam$ naturally gives  a fiber metric
 of $K_C^{-\half}$. Therefore, 
 the hermitian metric on the split vector bundle $E_0=K_C^{\frac{1}{2}}\oplus K_C^{-\frac{1}{2}}$ is given by the matrix
$h=\begin{bmatrix}
\lambda^{-1} &0\\
0&\lambda.
\end{bmatrix}$.
 Then $D$ becomes the Chern connection
 $D = D^{1,0}+D^{0,1}$, where
$
D^{0,1}=\bar{\partial}$ and
$ D^{1,0}=\partial+h^{-1}\cdot \partial h.
$
In terms of $\lam$, we have
$$
D=d+h^{-1}\cdot \partial h=d-\partial \log \lambda\cdot \begin{bmatrix}
1 &0\\
0&-1
\end{bmatrix} = d-\partial \log\lam\cdot H.
$$
\end{step}

\begin{step}
Let us introduce a Higgs field 
$$
\phi=\phi(q)=\begin{bmatrix}
0 & q\\
1& 0
\end{bmatrix}\cdot dz
=(X_-+q\cdot X_+) \cdot dz
$$ 
so that $(E_0,\phi(q))$ is on the Hitchin section,
where $q\in H^0\left(C,K_C^2\right)$.  The
hermitian conjugate of the Higgs field is 
calculated by 
$$
\phi^{\dagger_h}=h^{-1}\cdot \overline{\phi^{t}}\cdot h=
\begin{bmatrix}
\lambda &0\\
0&\lambda^{-1}
\end{bmatrix}\cdot
\overline{\begin{bmatrix}
0 & 1\\
q & 0
\end{bmatrix}\cdot dz} \cdot
\begin{bmatrix}
\lambda^{-1} &0\\
0&\lambda
\end{bmatrix}=
\begin{bmatrix}
0 & \lambda^2\\
\lambda^{-2} \cdot \overline{q}&0
\end{bmatrix} \cdot d\bar{z}
=\left(\lam^{-2}\cdot \bar{q}\cdot X_-+\lam^2 \cdot X_+\right)\cdot d\bar{z}.
$$
\end{step}

\begin{step}
Since we have identified all the ingredients, 
we can now write the
two-parameter family of 
connections in this local coordinate as
\begin{align*}
D(\zeta, R)&=\frac{R}{\zeta}\cdot \begin{bmatrix}
0 & q\\
1& 0
\end{bmatrix}\cdot dz+d-\partial \log \lambda\cdot \begin{bmatrix}
1 &0\\
0&-1
\end{bmatrix}\cdot dz+R \cdot \zeta \cdot \begin{bmatrix}
0 & \lambda^2\\
\overline{q}\cdot \lambda^{-2}& 0
\end{bmatrix}\cdot d\bar{z}
\\
&= d +\frac{R}{\zeta}\cdot dz \cdot (X_-+q \cdot X_+)
-\partial \log\lam\cdot  dz \cdot H+ R \cdot \zeta \cdot
d\bar{z} \cdot \left(\bar{q}\cdot\lam^{-2} \cdot
X_-+\lam^2 \cdot X_+\right).
\end{align*}
\end{step}

\begin{step}
A simple calculation shows that 
the coefficient of $H$ in the flatness condition 
$$
\left[D(\zeta, R), D(\zeta, R) \right]=0
$$ 
of $D(\zeta, R)$ yields
\begin{align*}
0 &= [d, -\partial \log\lam\cdot  dz \cdot H]
+R^2 \cdot dz\wedge d\bar{z}\cdot\left(\lam^2\cdot [X_-,X_+]
+q\cdot \bar{q}\cdot \lam^{-2}\cdot [X_+,X_-]\right)
\\
&=
-d\partial \log\lam\cdot dz \cdot H
+R^2 \cdot dz\wedge d\bar{z}\cdot \left(-\lam^2\cdot H
+q \cdot \bar{q} \cdot \lam^{-2} \cdot H\right)
\\
&=
\left(\bar{\partial}\partial\log \lam +R^2 \cdot (\lambda^{-2}\cdot q \cdot \overline{q} - \lambda^2)\right) \cdot dz\wedge d\bar{z}.
\end{align*}
Therefore, we obtain 
\be\label{harmonic}
\bar{\partial}\partial\log \lam +R^2 \cdot (\lambda^{-2}\cdot q \cdot \overline{q} - \lambda^2)= 0.
\ee
We thus conclude that the flatness condition for the two-parameter family of connections  $D(\zeta, R)$ gives the harmonicity condition 
\eqref{harmonic} for the hermitian metric $\lambda$.
\end{step}

\begin{step} For $q=0$, i.e., $\phi=X_{-}$,  the harmonicity equation \eqref{harmonic} becomes
\be\label{constant}
\bar{\partial}\partial\log \lam -R^2 \cdot \lambda^2 = 0,
\ee
which can be solved explicitly. We obtain
\be\label{lam 0}
\lam_0 = \frac{1}{R}\cdot \frac{i}{z-\bar{z}}.
\ee
Let us denote by 
\be\label{natural}
\lambda_\natural = \frac{i}{z-\bar{z}} = \frac{1}{2 \cdot y},
\ee
where $z= x+iy$.
The corresponding hermitian metric is then 
$$
g_\natural = \frac{dz\cdot d\bar{z}}{4\cdot y^2},
$$
whose Gaussian curvature is 
\be\label{curvature}
K := -\frac{4}{\lam_\natural^2} \cdot
\partial\bar{\partial}\log \lam_{\natural} = -4.
\ee
Indeed, $g_\natural$ is the globally defined
 constant curvature
metric on the upper half plane $\bH$, which is
invariant under the action of $PSL_2(\bR) = \Aut(\bH)$. 
Since we are dealing with a Riemann surface $C$
of genus $g\ge 2$, its universal covering is
$\bH$, and we have a non-canonical isomorphism
$C \isom \bH/\pi_1(C)$,
where $\pi_1(C)$ acts on $\bH$ through a 
representation $\rho:\pi_1(C)\lrar SL_2(\bR)$. 
By inducing a metric by the push-forward 
of the covering map $\bH\lrar C$, 
we conclude that the harmonicity equation
\eqref{constant} can be solved globally on $C$
with the hyperbolic metric on $C$ of
constant curvature $-4R^2$. 

Since $\frac{\zeta}{R}=\hbar$, we obtain
\be\label{q=0}
\begin{aligned}
D(\zeta, R)&=d+
\frac{1}{\hbar}\cdot \begin{bmatrix}
0 & 0\\
1& 0
\end{bmatrix}\cdot dz
-\partial \log \lambda_{\natural}\cdot \begin{bmatrix}
1 &0\\
0&-1
\end{bmatrix} dz+\hbar\cdot  \begin{bmatrix}
0 & \lambda_{\natural}^2\\
0& 0
\end{bmatrix} d\bar{z}
\\
&=
d+\frac{1}{\hbar}\cdot dz \cdot X_1-\partial \log \lambda_{\natural}\cdot dz \cdot H+
\hbar \cdot \lam_\natural^2 \cdot d\bar{z}\cdot X_+,
\end{aligned}
\ee
which \emph{does not} depend on $R$. 
\end{step}

\begin{step}
 The case when $q\neq 0$ in general.
 We remark that any hermitian metric compatible with the complex structure  of  $C$ is conformal to the 
 constant curvature metric $g_\natural$. 
 Therefore, we can write 
$$
\lambda=\lambda_{0}\cdot e^{f(R)}=\frac{\lambda_\natural}{R}\cdot e^{f(R)}
$$
with a conformal factor $e^{f(R)}$ depending
on a real
valued function $f(R)$ on $C$. We plug this
expression into \eqref{harmonic} and apply
the 
 implicit function theorem to yield that $f$ is real analytic, and more significantly, that  $f(R)=f_4\cdot R^4+$ \textit{higher order terms}. This implies
$$
\lambda^{-2}=\frac{R^2}{\lambda_\natural^2}\cdot e^{-2f(R)}=\frac{R^2}{\lambda_\natural^2}+\mathcal{O}(R^6),
\quad  \lambda=\frac{\lambda_\natural}{R}+\mathcal{O}(R^3), \quad 
\partial \log \lambda=\partial \log \lambda_\natural+ \cO(R^4).
$$
Therefore, we obtain the scaling limit as
\be\label{general}
\begin{aligned}
D(\zeta,R)&=d+\frac{1}{\hbar}\cdot \begin{bmatrix}
0 & q\\
1& 0
\end{bmatrix}\cdot dz-\partial \log \lambda\cdot \begin{bmatrix}
1 &0\\
0&-1
\end{bmatrix}\cdot dz+R^2 \cdot\hbar \begin{bmatrix}
0 & \lambda^2\\
\overline{q}\cdot \lambda^{-2}& 0
\end{bmatrix}\cdot d\bar{z}\\
&=
d+ \frac{1}{\hbar}\cdot \begin{bmatrix}
0 & q\\
1& 0
\end{bmatrix}\cdot dz-\partial \log \lambda_{\natural}\cdot \begin{bmatrix}
1 &0\\
0&-1
\end{bmatrix}\cdot dz
+
\cO(R^4)\cdot \begin{bmatrix}
1 &0\\
0&-1
\end{bmatrix}\cdot dz
\\
&+\hbar \cdot \begin{bmatrix}
0 & \lambda_\natural^2+\mathcal O(R^4)\\
\overline{q}\cdot\frac{R^4}{\lambda_\natural^2}+
 \mathcal{O}(R^8) & 0
\end{bmatrix}\cdot d\bar{z}
\\
&\frac{R\rar 0, \zeta\rar 0}{\zeta/R = \hbar}\!\!\!\rar
d+\frac{1}{\hbar}\cdot \begin{bmatrix}
0 & q\\
1& 0
\end{bmatrix}\cdot dz-\partial \log \lambda_{\natural}\cdot \begin{bmatrix}
1 &0\\
0&-1
\end{bmatrix} \cdot dz+\hbar \cdot \begin{bmatrix}
0 & \lambda_{\natural}^2\\
0& 0
\end{bmatrix} \cdot d\bar{z}
\\
&=
d+\frac{1}{\hbar}\cdot dz \cdot (X_-+q \cdot X_+) -
\partial \log \lambda_{\natural}\cdot dz \cdot H+
\hbar \cdot \lam_\natural^2 \cdot d\bar{z} \cdot X_+.
\end{aligned}
\ee
We can see that the only dependence on 
the quadratic differential $q$ is
in the form of $\phi(q)$.
\end{step}


\section{The limit oper of Gaiotto's correspondence and the quantum curve}
In this section we will prove that the 
scaling limit $\lim_{R, \zeta\rightarrow 0, \frac{\zeta}{R}=\hbar}D(\zeta, R)$ is an oper.

A surprising fact is  that the limit oper of Theorem \ref{DFKMMN} \cite{OD20} is gauge equivalent to $\nabla^{\hbar}=d+\frac{1}{\hbar} \cdot \phi$ in 
the M\"obius coordinate system obtained by 
the uniformization of the base curve $C$. In other words, the limit oper of Gaiotto correspondence in rank two  is the quantum curve of Theorem~\ref{quantum} (\cite{OD8}). We use
Theorem~\ref{gun1} and Lemma~\ref{gun2} of Section~\ref{Deligne} to imply that the scaling limit of Gaiotto's correspondence is actually an oper. For the sake of completeness, we include detailed computations here.

\begin{prop}[Gauge transform of the scaling Limit]
The limit expression of 
\eqref{general},
$$
D(\hbar):=d+\frac{1}{\hbar}\cdot dz \cdot (X_-+q \cdot X_+) -
\partial \log \lambda_{\natural}\cdot dz \cdot H+
\hbar \cdot\lam_\natural^2 \cdot d\bar{z} \cdot X_+,
$$
is gauge equivalent to an oper
\be \label{nabla}
\nabla^{\hbar}:=d+\frac{1}{\hbar}\cdot \begin{bmatrix}
0 & q\\
1& 0
\end{bmatrix}\cdot dz = d+\frac{1}{\hbar}\cdot \phi(q).
\ee
Here, $\lambda_{\natural}=\frac{i}{z-\bar{z}}$
is the local expression of the hermitian metric of 
constant curvature $-4$
 on $C$. If we introduce a M\"obius 
 coordinate system on $C$ induced by the
 uniformization covering $\bH\lrar C$, then
 \eqref{nabla} determines a globally defined
 oper on $C$. 
\end{prop}

\begin{proof}
The claim is that
$\nabla^{\hbar}=g \cdot D(\hbar) \cdot g^{-1}$ 
with the gauge transformation 
$$
g=\begin{bmatrix}
1 & \hbar\cdot \partial \log \lambda_{\natural}\\
0 & 1
\end{bmatrix}
=e^{\hbar \cdot \partial\log\lam_\natural\cdot  X_+}.
$$
For brevity, let us denote $a=\partial \log \lambda_{\natural}$ and $b=\lambda_{\natural}^2$.
Then
$$
D(\hbar) =d+ \frac{1}{\hbar} \cdot dz \cdot
(X_-+q \cdot X_+)-a\cdot dz \cdot H +\hbar \cdot b \cdot d\bar{z} \cdot X_+,
\qquad g =   e^{\hbar \cdot a \cdot X_+}.
$$ 
Recall that for any elements $A$ and $B$ of a Lie 
algebra and a central parameter $\epsilon$, we have 
the adjoint formula
$$
e^{\epsilon A} \cdot B \cdot e^{-\epsilon A}
=\sum_{n=0}^\infty \frac{1}{n!} \epsilon^n \cdot
\ad^n_A (B),
$$
where
$$
\ad^n_A (B): =\overset{n}
{\overbrace{[A,[A,[\cdots,[A}},B]\cdots]]].
$$
From the commutation relations of the basis for
$sl_2(\bC)$ , 
we see that 
$$
g \cdot X_- \cdot g^{-1} = X_- + \hbar \cdot a \cdot [X_+,X_-]
+ \half (\hbar \cdot a)^2 \cdot [X_+,[X_+,X_-]]
=X_-+ \hbar \cdot a \cdot H - (\hbar \cdot a)^2 \cdot X_+.
$$
Therefore, 
$$
g\cdot \frac{1}{\hbar} \cdot dz \cdot
(X_-+q \cdot X_+)\cdot g^{-1}=
 \left(\frac{1}{\hbar}\cdot (X_-+q \cdot X_+)
 +a\cdot H -\hbar \cdot a^2 \cdot X_+
\right)\cdot dz.
$$
Similarly, 
$$
g\cdot H \cdot g^{-1} = H +\hbar\cdot a \cdot [X_+,H]  =H -2\hbar \cdot a \cdot X_+,
$$
which yields
$$
-g\cdot a \cdot dz\cdot H\cdot g^{-1} =- a \cdot dz \cdot H + 2\hbar \cdot a^2 \cdot dz \cdot X_+.
$$
It is obvious that 
$$g\cdot \hbar \cdot b \cdot  d\bar{z} \cdot  X_+\cdot 
g^{-1} = \hbar \cdot b \cdot d\bar{z} \cdot X_+.
$$
Finally, we calculate
\begin{align*}
g\cdot d\cdot g^{-1} &= d -\hbar \cdot da \cdot X_+ 
\\
&=d-\hbar \cdot (\partial a \cdot dz+\bar{\partial} \cdot a \cdot d\bar{z}) \cdot X_+
\\
&= d-\hbar \cdot
dz \cdot \partial^2 \log \lam_\natural \cdot X_+
-\hbar \cdot d\bar{z} \cdot \bar{\partial}\partial \log \lam_\natural \cdot
X_+
\\
&=d +\hbar \cdot dz \cdot \frac{(\partial \lam_\natural)^2 -
\lam_\natural \cdot \partial^2\lam_\natural}{\lam_\natural^2} \cdot X_+
- \hbar \cdot d\bar{z} \cdot \lam_\natural^2 \cdot X_+,
\end{align*}
where we used the constant curvature property
\eqref{curvature} in the last step. 

Adding all together, we obtain the gauge 
transformation formula
\begin{align*}
g \cdot D(\hbar)\cdot g^{-1} &=
d + \left(\frac{1}{\hbar} \cdot (X_-+q \cdot X_+)
 +\partial \log\lam_\natural \cdot H -\hbar  \cdot
 (\partial\log\lam_\natural)^2 \cdot X_+
\right) \cdot dz+
\\
&\qquad
+ \hbar \cdot dz \cdot \frac{(\partial \lam_\natural)^2 -
\lam_\natural \cdot \partial^2\lam_\natural}{\lam_\natural^2} \cdot X_+
+ (\hbar \cdot b \cdot d\bar{z}-\hbar \cdot d\bar{z} \cdot \lam_\natural^2) \cdot
 X_+ -\\
&\qquad
-\partial \log \lam_\natural \cdot dz \cdot H +
2 \hbar \cdot (\partial \log\lam_\natural)^2\cdot dz \cdot X_+
\\
&=
d + \frac{1}{\hbar} \cdot (X_-+q \cdot X_+)\cdot dz 
+\hbar \cdot dz \cdot \frac{2 \cdot(\partial \lam_\natural)^2 -
\lam_\natural \cdot \partial^2\lam_\natural}{\lam_\natural^2} \cdot X_+
\\
&=
d+\frac{1}{\hbar}\cdot \phi(q).
\end{align*}
Here, we use the fact that
$2 \cdot (\partial \lam_\natural)^2 -
\lam_\natural \cdot \partial^2\lam_\natural=0$, which 
follows from the expression \eqref{natural}.
\end{proof}

The computation we have shown above 
is only for $SL_2(\bC)$. Yet it is valid for
proving vastly general Theorem~\ref{DFKMMN}. 
The key idea is to use Kostant's principal TDS
\cite{Kostant}, replacing the basis 
$\la H,X_\pm\ra$ by the one for TDS.
Then almost  exactly the  same formulas hold for the
general situation. Here, we only
indicate the oper we obtain 
through Gaiotto's scaling limit for the case of
$SL_r(\bC)$-Higgs bundles. We will give a geometric definition of \textit{oper} that generalizes Gunning's Definition encaptured by Theorem \ref{gun1} from rank two to arbitrary rank $r$.

\begin{Def}[Beilinson-Drinfeld, 1993]
\label{SLr oper}
 Let $V$ be a holomorphic vector bundle of rank $r$
and degree $0$. An $SL_r(\mathbb{C})$-oper is a pair $(V, \nabla)\in \cM_{\deR}$ satisfying the 
following conditions:
\begin{enumerate}
\item There is a global filtration
$0=F_{r}\hookrightarrow F_{r-1}\hookrightarrow \ldots \hookrightarrow F_{0}=V$ in $V$.

\item Griffiths transversality.
The connection $\nabla$ induces a map
$\nabla|_{F_{i+1}}:F_{i+1}\rightarrow F_{i}\otimes K_{C}$  for every $i=0, \dots,r-1$.

\item $\nabla|_{F_{i+1}}$ 
induces  an $\mathcal{O}_{C}$-linear isomorphism $F_{i+1}/F_{i+2}\cong F_{i}/F_{i+1} \otimes K_C$
for every $i$.
\end{enumerate}
\end{Def}

Let $(E_0,\phi(q))$ be a point on the
Hitchin section of Section~\ref{hitchin section}.
Then Gaiotto's scaling limit produces an 
$\hbar$-family of opers $(V_{\hbar}, \nabla^{\hbar})$
defined as follows. First, we choose once and 
for all the M\"obius coordinate system
associated with the uniformization mentioned above.

\begin{itemize}

\item $V_{\hbar}$ is given by 
the transition function
$\{f_{\a\b}^{\hbar}\}$, where 
$$
f_{\a\b}^{\hbar}=\exp(H \cdot \log\xi_{\a\b}) \cdot
\exp\left(\hbar \cdot \frac{d\log \xi_{\a\b}}{dz_{\b}}\cdot X_{+}\right).
$$

\item The connection is defined by
$$
\nabla^\hbar := d + \frac{1}{\hbar} \cdot \phi(q).
$$
Note that this definition is globally valid with
respect to a M\"obius coordinate system.

\item $V_{0}=K_C^{\frac{r-1}{2}}\oplus \ldots \oplus K_C^{-\frac{r-1}{2}}=E_0$,  since $f_{\a\b}^{\hbar=0}=\exp(H\cdot \log\xi_{\a\b})=\xi_{\a\b}^{H}$.

\item There is a unique filtration in the vector bundle $V_{\hbar}$ with $F_{r-1}=K_C^{\frac{r-1}{2}}$
that satisfies the conditions  of 
Definition~\ref{SLr oper}.

\item The vector bundles $V_{\hbar}$ 
are isomorphic for all $\hbar\ne 0$.

\end{itemize}

We refer to \cite{OD21,OD22} for more detail.

\section{Conclusion} 

We emphasize again  that the nonabelian Hodge correspondence is a diffeomorphism between 
$\cM_{\Dol}$ and $\cM_{\deR}$. 
 (\cite{D,H1,S}). Conjecture \ref{dgaiotto} of
Gaiotto  realizes a \textbf {holomorphic point by point correspondence} between two holomorphic Lagrangians,  the {\it Hitchin section} 
in $\cM_{\Dol}$ and the  {\it moduli space of opers}
in $\cM_{\deR}$. Since the 
quantum curve should depend holomorphically
on the spectral curve, we consider 
the Gaiotto correspondence as the
desired  construction 
of quantum curves. 

\begin{itemize}
\item {\color{red} Donaldson, Hitchin, Simpson} Nonabelian-Hodge correspondence

$${\color{blue} \mathcal{M}_{\text{Dol}}\stackrel{\text{diffeomorphism}}{\longrightarrow}\mathcal{M}_{\text{deR}}}$$

\item {\color{red} Gaiotto's conjecture} gives a holomorphic correspondence between Lagrangians 

$$
\text{\color{blue} Hitchin Section $\stackrel{\text{holomorphic}}{\longrightarrow}$ moduli space of opers}
$$
\end{itemize}

Quantization is never unique. Yet the Catalan 
example we presented earlier clearly 
shows why we are interested in 
the unique process of quantization. A quantum curve
quantizes the B-model geometry, which provides
a generating function of genus $g$ A-models for
all $g$. Thus we wish a unique quantization result.

Starting from a Hitchin spectral curve, we identify the
Higgs bundle on the Hitchin section. This is 
unique, once the spin structure of the curve
$C$ is chosen. Then the correspondence 
given by Gaiotto's scaling limit constructs, again, a 
unique oper in the moduli space of holomorphic
connections on $C$. Thus the process from the
spectral curve to the quantum curve (oper) is unique, and
depends holomorphically on the moduli of spectral
curves, when the complex structure of $C$ is 
fixed.

We present below a local
picture of the two Lagrangians inside the Dolbeault moduli space together with their images under the nonabelian-Hodge and Gaiotto's correspondences.
In the figure, $V_1 = V_\hbar|_{\hbar = 1}$.
The picture does not show the global 
relations between various Lagrangians in the
moduli spaces. For example, the $SL_r(\bR)$-Hitchin
component in $\cM_{\deR}$ and the oper moduli 
space intersects infinitely many times. Only locally
they intersect at a point, here at $(V_1,d+X_-)$. 

\vspace{0.5 in}
\begin{figure}[htb]

\includegraphics[width=4.8in]{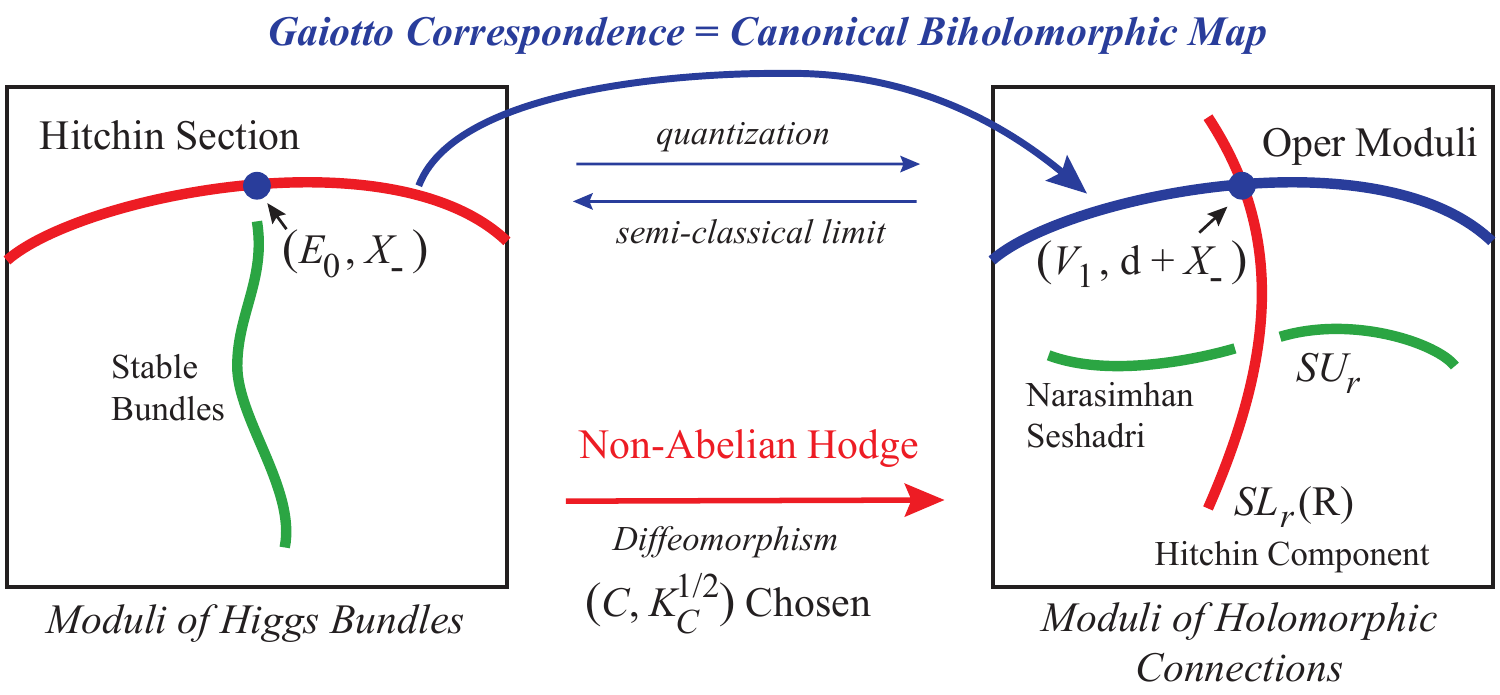}

\end{figure}

\vfill


\bibliographystyle{amsplain}

\begin{thebibliography}{10}

\bibitem{ADKMV} M.~Aganagic, R.~Dijkgraaf, A.~Klemm, M.~Mari\~no, and C.~Vafa,
\emph{Topological Strings and Integrable Hierarchies},[arXiv:hep-th/0312085], 
Commun.\ Math.\ Phys.\ \textbf{261}, 451--516 (2006).


 \bibitem{At} M.F.~Atiyah, \emph{Vector bundles over an elliptic curve}, Proc. London Math. Soc. (3), 7:414–452, 1957.


 
 \bibitem{AB}
M.F.~Atiyah and R.~Bott, 
\emph{The Yang-Mills equations over Riemann surfaces},
Phil. Trans. Royal Soc. London A \textbf{308}, 523--615  (1982).

\bibitem{BD} A.~Beilinson and V.~Drinfeld, \emph{Opers}, arXiv:math/0501398v1 [math.AG] (2005).




\bibitem{DHS} R.~Dijkgraaf, L.~Hollands, and P.~Su\l kowski, \emph{Quantum curves and $\cD$-modules}, Journal of High Energy Physics \textbf{0810.4157}, 1--58 (2009).
	
\bibitem{DHSV} R.~Dijkgraaf, L.~Hollands P.~Su\l kowski,
and C.~Vafa, \emph{Supersymmetric gauge theories, intersecting branes and free Fermions},
Journal of High Energy Physics \textbf{0802.106},  (2008).


\bibitem{DVV} R.~Dijkgraaf, E.~Verlinde, and H.~Verlinde,
\emph{Loop equations and {V}irasoro constraints in non-perturbative two-dimensional quantum gravity}, Nucl. Phys. \textbf{B348}, 435--456 (1991).

 



\bibitem{D} S.K.~Donaldson, \emph{A new proof of a theorem of Narasimhan and Seshadri}, Journal of Differential Geometry, 18 (2): 269–277, (1983).


\bibitem{OD20} O.~Dumitrescu L.~Fredrickson, G.~Kydonakis, R.~Mazzeo, M.~Mulase and A.~Neitzke, {\it Opers versus nonabelian Hodge} http://arxiv.org/pdf/1607.02172v1.pdf, under review.

\bibitem{OD8} O.~Dumitrescu and M.~Mulase,
{\it Quantum curves for Hitchin fibrations and the Eynard-Orantin theory}, Lett.\ Math.\ Phys.\ \textbf{104}, 635--671 (2014).

\bibitem{OD12} O.~Dumitrescu and M.~Mulase, {\it Quantization of spectral curves for meromorphic Higgs bundles through topological recursion}, http://arxiv.org/pdf/1411.1023v1.pdf, under review.

\bibitem{OD16} O.~Dumitrescu and M.~Mulase, {\it Edge-contraction on dual ribbon graphs, 2D TQFT, and the mirror of orbifold Hurwitz numbers},
http://arxiv.org/pdf/1508.05922v1.pdf, under review in {\it Journal of Algebra}.

\bibitem{OD17} O.~Dumitrescu and M.~Mulase, {\it Lectures on the topological recursion for Hitchin spectral curves and quantization}, http://arxiv.org/pdf/1509.09007v1.pdf to appear in Lecture Notes Series, Institute for Mathematical Sciences, National University of Singapore.


\bibitem{OD21} O.~Dumitrescu and M.~Mulase, \emph{Lectures on topological quantum field theory and character varieties}, Lecture Notes.

\bibitem{OD22} O.~Dumitrescu and M.~Mulase, \emph{Weyl quantization of Hitchin spectral curves and opers}, ``Proceedings of the 2016  AMS von Neumann Symposium," Proceedings of Symposia in Pure Mathematics, American Mathematical Society. 

\bibitem{OD6} O.~Dumitrescu, M.~Mulase, B.~Safnuk, A.~Sorkin {\it The Spectral Curve of the Eynard-Orantin Recursion via the Laplace Transform} in Algebraic and Geometric Aspects of Integrable Systems and Random Matrices, Dzhamay, Maruno and Pierce, Eds. Contemporary Mathematics 593, 263--315 (2013)

\bibitem{EO}
  B.~Eynard and N.~Orantin,
\emph{Invariants of algebraic curves and 
topological expansion},
Communications in Number Theory
and Physics {\bf 1},  347--452 (2007).

\bibitem{G} D.~Gaiotto,\emph{Opers and TBA}, arXiv:1403.6137 [hep-th], (2014).

\bibitem{GMN} D.~Gaiotto, G.W.~Moore, and A.~Neitzke,\emph{Wall-crossing, Hitchin systems, and the WKB approximation}, arXiv:0907.3987 [hep-th] (2009).		



\bibitem{Gr} A.~Grothendieck,  \emph{Sur la classification des fibres holomorphes sur la sph\`ere de Riemann}, Am. J. Math., 79:121–138, 1957



\bibitem{Gro}
A.~Grothendieck,
\emph{Esquisse d'un programme}, (1984).

\bibitem{GS} S.~Gukov and P.~Su\l kowski, \emph{A-polynomial, B-model, and quantization}, arXiv:1108.0002v1 [hep-th] (2011).

\bibitem{Gun} R.C.~Gunning, \emph{Special coordinate coverings of Riemann surfaces}, Math.\ Annalen \textbf{170}, 67--86 (1967).


\bibitem{Harer} J.~L. Harer, \emph{The cohomology of the moduli space of curves},
in Theory of Moduli, Montecatini Terme, 1985 (Edoardo Sernesi, ed.), Springer-Verlag,
 1988, pp.~138--221.

\bibitem{HZ} J.~Harer and D.~Zagier,\emph{The Euler characteristic of the moduli space of curves}, Inventiones Mathematicae \textbf{85}, 457--485 (1986).

\bibitem{H1} N.J.~Hitchin, \emph{The self-duality equations on a Riemann surface}, Proc.\ London Math.\ Soc.\ (Ser.\ 3) \textbf{55}, 59--126  (1987).


\bibitem{K1992}
M.~Kontsevich, \emph{Intersection theory on the moduli space of curves and the
matrix {Airy} function}, Communications in Mathematical Physics \textbf{147}, 1--23  (1992).

\bibitem{KM} M.~Kontsevich and Y.~Manin, \emph{Gromov-Witten classes, quantum 
cohomology, and enumerative geometry}, Communications in Mathematical Physics \textbf{164}, 525--562 (1994).




 \bibitem{Kostant}
 B.~Kostant,
 \emph{The principal three-dimensional subgroup and the Betti numbers of a complex simple Lie group},
 Amer.\ J.\ Math., \textbf{81}, 973--1032 (1959).


 \bibitem{MP1998} M.~Mulase and M.~Penkava, \emph{Ribbon graphs, quadratic differentials on Riemann surfaces, and algebraic curves defined over $\overline{\mathbb{Q}}$}, The Asian Journal of Mathematics  \textbf{2} (4), 875--920 (1998).

\bibitem{MS} M.~Mulase and P.~Su\l kowski, \emph{Spectral curves and the Schr\"odinger equations for the Eynard-Orantin recursion}, arXiv:1210.3006 [math.ph](2012).




\bibitem{Mu} D.~Mumford, \emph{Projective invariants of projective structures and applications}, Proc. Internat. Congr. Mathematicians (Stockholm, 1962), pages 526–530. Inst. Mittag-Leffler, Djursholm, 1963.

\bibitem{Mumford} D.~Mumford, \emph{Towards an enumerative geometry of the moduli space of curves}, in ``Arithmetic and Geometry,'' M.~Artin and J.~Tate, eds., Part II, 271--328, Birkh\"auser, 1983.

\bibitem{MFK}
D.~Mumford, J.~Fogarty, and F.~Kirwan,
\emph{Geometric invariant theory},
Ergebnisse der Mathematik und ihrer Grenzgebiete, 294 pages,
Springer-Verlag (1994).


\bibitem{NS} 
M.S.~Narasimhan and C.S.~Seshadri, 
\emph{Stable and unitary vector bundles on a compact Riemann surface}, Annals of Mathematics, Second Series, \textbf{82}, 540--567 (1965).



\bibitem{OP} A.~Okounkov and R.~Pandharipande, \emph{Gromov-Witten theory, Hurwitz numbers, and matrix 
models, I}, Proc.\ Symposia Pure Math textbf{80}, 325--414 (2009).


\bibitem{Se} C. S.~Seshadri, \emph{Space of unitary vector bundles on a compact Riemann surface}, Annals of Mathematics,  Series 2, \textbf{85}, 303--336 (1967).

\bibitem{S} C.T.~Simpson, \emph{Higgs bundles and local systems}, Publications Math\'ematiques de l'I.H.E.S.\ \textbf{75}, 5--95 (1992).

\bibitem{STT} D.~D. Sleator, R.~E. Tarjan, and W.~P. Thurston, \emph{Rotation distance, triangulations, and hyperbolic geometry}, Journal of the American Mathematical Society \textbf{1}, 647--681 (1988).

\bibitem{Strebel} K.~ Strebel, \emph{Quadratic differentials}, Springer-Verlag, 1984.

\bibitem{tH} G.~'t Hooft, \emph{ A planer diagram theory for strong interactions},
Nuclear Physics \textbf{B72}, 461--473 (1974).


\bibitem{WL} T.R.S.~Walsh and A.B.~Lehman, \emph{Counting rooted maps by genus. I}, Journal of Combinatorial Theory 
\textbf{B-13}, 192--218 (1972).




  



\end{thebibliography}

\end{document}